\documentclass{amsart}
\usepackage{latexsym,amsmath,amssymb,mathrsfs,pstricks,pst-node,epsfig}
\renewcommand{\wr}{\mathop{\textrm{wr}}}
\newcommand{\Sym}{\mathop{\textrm{Sym}}}
\newcommand{\SL}{\mathop{\textrm{SL}}}
\newcommand{\Aut}{\mathop{\textrm{Aut}}}
\newcommand{\PSL}{\mathop{\textrm{PSL}}}
\newcommand{\Alt}{\mathop{\textrm{Alt}}}
\renewcommand{\H}{\mathop{\mathrm{H}}}

\newtheorem{theorem}{Theorem}[section]
\newtheorem{lemma}[theorem]{Lemma}
\newtheorem{example}[theorem]{Example}
\newtheorem{remark}[theorem]{Remark}
\newtheorem{definition}[theorem]{Definition}
\newtheorem{proposition}[theorem]{Proposition}
\newtheorem{corollary}[theorem]{Corollary}

\begin{document}
\title[$\wedge$-transitive graphs and cartesian decompositions]{$\wedge$-transitive digraphs preserving a cartesian decomposition}
\author[J.~Morris]{Joy Morris} 
\address{Department of Mathematics and Computer Science \\
University of Lethbridge \\
Lethbridge, AB. T1K 3M4. Canada}
\email{joy.morris@uleth.ca}

\author[P. Spiga]{Pablo Spiga}
\address{Pablo Spiga, Dipartimento di Matematica Pura e Applicata,\newline
University of Milano-Bicocca,
Via Cozzi~53, 20126 Milano Italy}  \email{pablo.spiga@unimib.it}

\thanks{This research was supported in part by the National Science
  and Engineering Research Council of Canada} 



\begin{abstract}
In this paper, we combine group-theoretic and combinatorial techniques to study $\wedge$-transitive digraphs admitting a cartesian decomposition of their vertex set. In particular, our approach uncovers a new family of digraphs that may be of considerable interest.
\end{abstract}
\maketitle

\section{Introduction}\label{intro1}
One of the most interesting families of highly symmetric graphs is the family of {\em distance-transitive} graphs. We refer the reader to the survey article~\cite{vanBon} for the current status of the project of classifying these graphs. A major step towards this classification is a theorem of Praeger, Saxl and Yokoyama~\cite{PSY}, which investigates the structure of  distance-primitive graphs:  distance-transitive graphs admitting a group of automorphisms acting primitively on the vertices. (A graph is distance-transitive if its automorphism group acts transitively on ordered pairs of vertices at each fixed distance.) The main tool in~\cite{PSY} is the O'Nan-Scott theorem for finite primitive permutation groups.

% In~\cite{PSY}, the Classification of Finite Simple Groups is used first for dealing with primitive groups of product action type and then via the Schreier conjecture (for dealing with primitive groups of twisted wreath type). 
%
%In this paper, we remove the dependency of \cite{PSY} on the Classification of the Finite  Simple Groups, in the case of product action. Our proof can be used to replace the proofs in Section 2 of~\cite{PSY}, and does not use the CFSG.  It should be mentioned that in~\cite{PSY} the authors say they have found a Classification-free proof of the results in Section 2 in the case where ``$\Gamma$ is self-paired" - i.e., the graph is not a digraph - but they do not provide even a sketch of this proof.  In this paper, we provide such a proof and, in addition, are able to deal with the case of digraphs without resorting to the CFSG. The full result of~\cite{PSY} also uses the CFSG (specifically, the Schreier conjecture) to deal with primitive groups of twisted wreath type, which we do not consider in this paper.  However, ours is a major step towards removing~\cite{PSY}'s dependence on the Classification.

In this paper, by relaxing the conditions required of a distance-primitive graph in three ways, we discover a hitherto-unstudied family of digraphs ((iii) from Theorem~\ref{main-thm} below).  These digraphs might well provide nice examples or counter-examples to other interesting problems.  Their automorphism groups are not transitive on directed paths between vertices at distance 2 (except when $n=1$), but are transitive on the arcs of the digraph, with a lot of additional symmetry.  Our three relaxations are:
first, we allow directed graphs.  Second, our (di)graphs admit a kind of 2-distance-transitivity, but we do not require any higher distance-transitivity. More precisely, our (di)graphs are transitive on ordered pairs of vertices that are either adjacent, or are non-adjacent but share an out-neighbour. In the case of graphs, this is the same as 2-distance-transitivity.  Third, our (di)graphs do not necessarily admit a group of automorphisms acting primitively on their vertex set. It should be noted that the (di)graphs we consider must admit a group of automorphisms that preserves a cartesian decomposition, and cannot be Cayley (di)graphs on a specified subgroup of their automorphism group (see Section~\ref{intro} and Definition~\ref{def} for a precise statement of our hypothesis).  Additional examples do arise if we allow Cayley (di)graphs; Lemma~\ref{tournamentprel}, for example, shows that Payley tournaments are one such class; cycles (directed or undirected) are another.

In~\cite{PSY}'s analysis of distance-primitive graphs, the authors use the Classification of Finite Simple Groups in two ways: first, in their Proposition 2.4, to deal with primitive groups of product action type, and then (via the Schreier conjecture) to deal with primitive groups of twisted wreath type.  Proposition 2.4 of~\cite{PSY} is about digraphs, and while Classification-free proofs are known in the case where the graphs are undirected (one appears in~\cite{vanBon thesis}), they are not readily available. Neither \cite{PSY} nor \cite{vanBon thesis} claim knowledge of a Classification-free proof in the directed case.  We provide such a proof in this paper.  Naturally, since our paper is about product action, we do not consider the use of the Classification in \cite{PSY}'s analysis of twisted wreath type. 

Our main result is the following.  The notation $\H(m,n)$ is used for the Hamming graph that is isomorphic to the cartesian product of the complete graph $K_m$ with itself $n$ times.
\begin{theorem}\label{main-thm}Let $G$ be a product action type group with base $N$. Suppose that $G$ acts
  transitively on the arcs of the digraph $\Gamma$, and on the pairs of vertices that are non-adjacent but share an out-neighbour. If $\Gamma$ is not a Cayley digraph on $N$ then $\Gamma$ is isomorphic
  to one of:
  \begin{itemize}
  \item[(i)] $\H(m,n)$;
  \item[(ii)] the complement of $\H(m,2)$; or 
  \item[(iii)] one of the graphs $X_q(n)$ in Example~\ref{Ex2}. In this case, $\Gamma$ is a
  digraph.
  \end{itemize}
\end{theorem}

In the light of Theorem~\ref{main-thm} and the proof of the main theorem of~\cite{PSY} we see (without appealing to the Classification of Finite Simple Groups) that  if $\Gamma$ is a distance-primitive graph, then either $\Gamma$ is as in Theorem~\ref{main-thm}~(i) and~(ii), or $\Gamma$ is a Cayley graph over a characteristically simple group.

\begin{corollary}$\H(m,2)$ and its complement are the only $2$-arc-transitive graphs admitting a group of automorphisms of product action type.
\end{corollary}

The hypothesis on  the decomposition of the vertex set of $\Gamma$ as a cartesian product is important. In fact, Li and Seress~\cite{LS} have obtained several intricate examples of $2$-distance transitive graphs $\Gamma$ with $V\Gamma$ not admitting an $\Aut(\Gamma)$-invariant cartesian decomposition. In these remarkable examples $V\Gamma$ has a $\Aut(\Gamma)$-invariant partition $\mathcal{B}$, and the quotient graph $\Gamma/\mathcal{B}$ does admit a cartesian decomposition.

Finally, the definition of ``product action type'' that we use  (see Definition~\ref{def}) is inspired by~\cite{PS}, which is a complete treatment of  permutation groups that preserve a cartesian decomposition.

\section{Notation and basic examples}\label{intro}
Let $H$ be a permutation group acting on the set $\Delta$, let $T$ be a
transitive normal subgroup of $H$ and let $K$ be a transitive subgroup of
the symmetric group $\Sym(n)$ on $\{1,\ldots,n\}$ with $n\geq 2$. We let
$W$ denote the wreath 
product $H\wr K$ acting on the cartesian product $\Omega=\Delta^n$. Thus, for $\sigma\in K$ and $h_1,\ldots,h_n\in H$, the group
element $g=\sigma(h_1,\ldots,h_n)$ of $W$ acts on $(\delta_1,\ldots,\delta_n)\in\Omega$ by    
\[(\delta_1,\ldots,\delta_n)^g=
(\delta_{1^{\sigma^{-1}}}^{h_1},\ldots,\delta_{n^{\sigma^{-1}}}^{h_n}).\] 
In other words, $\sigma$ permutes the $n$ coordinates  and the $n$-tuple $(h_1,\ldots,h_n)$ acts coordinate-wise. 

For $i\in \{1,\ldots,n\}$, we denote by $T_i$ the $i$th
coordinate subgroup of $T^n$, that is, $T_i=\{(t_1,\ldots,t_n)\in
T^n : t_j=1,\,\forall j\in\{1,\ldots,n\}\setminus\{i\}\}$. As $H$ normalizes
$T$, the group $W$ acts by 
conjugation on the set $\{T_1,\ldots,T_n\}$ and the action of $W$ on $\{T_1,\ldots,T_n\}$ is permutation
equivalent to the action of $K$ on $\{1,\ldots,n\}$. Furthermore, the
normal subgroup $N=T_1\times \cdots \times T_n$ of $W$ acts
transitively on $\Omega$.

\begin{definition}\label{def}{\rm We say that $G\leq W$ is of \emph{product action type} with base $N$  if
\begin{enumerate}
\item[(i)]$T$ is not regular on $\Delta$,
\item[(ii)]$N\leq G$, and 
\item[(iii)]the action of $G$ on $\{T_1,\ldots,T_n\}$ is
transitive.
\end{enumerate}}
\end{definition} For each $i\in \{1,\ldots,n\}$, consider the group $G_i=N_G(T_i)$. If $g=\sigma(h_1,\ldots,h_n)\in G_i$, then $i^\sigma=i$ and the function
$\pi_i:G_i\to H$  mapping $g$ to $h_i$  defines a group
homomorphism. As $G$ is transitive on $\{1,\ldots,n\}$, for each
$i,j\in \{1,\ldots,n\}$, the group $H^{\pi_i}$ is conjugate to
$H^{\pi_j}$. In particular, replacing $H$ by the image of $\pi_i$ if
necessary, we may assume that each $\pi_i$ is surjective, for each $i$.

In this paper, we assume that $G$ is of product action type and is a group of automorphisms of a
connected (directed or undirected) graph $\Gamma$ with vertex set
$V\Gamma=\Omega$. We let $A\Gamma$ denote the arcs of $\Gamma$ and, for a
vertex $v$ of $\Gamma$, we let $\Gamma^+(v)$ (respectively
$\Gamma^-(v)$) denote the out-neighbours (respectively in-neighbours)
of $v$ and we write   
\[
A^2_+\Gamma=\{(u,v)\in V\Gamma\times V\Gamma : u,v \textrm{ are 
non-adjacent and }
 \Gamma^+(u)\cap \Gamma^+(v)\neq \emptyset\}\]
(equivalently, $(u,v)\in A^2_+\Gamma$ if $u$ and $v$ are non-adjacent and $u,v\in \Gamma^{-}(w)$, for some $w\in V\Gamma$).

We are concerned with the following kind of action.

\begin{definition}
A group $G$ acts $\wedge$-{\em transitively} on a digraph $\Gamma$ if $G$ acts transitively on $A\Gamma$ and on
$A^2_+\Gamma$. (We also say that $\Gamma$ is $\wedge$-transitive.)
\end{definition}

If $\Gamma$ is undirected, then our
definition coincides with the definition of
$2$-distance-transitive graphs. For digraphs,
this is not the most natural definition of $2$-distance-transitivity, hence we use the term $\wedge$-transitive for this action, a term that was suggested to the first author by Peter Neumann during a discussion of this work.
In generalising our arguments from the undirected case to the directed case, this is the transitivity requirement that most naturally arises.

We have two motivations (aside from
feasibility) for using this definition. 
First, this sort of $2$-distance-transitivity has been previously studied, although not named. It was investigated by Praeger, Saxl and Yokoyama in~\cite{PSY} (see for example~\cite[Proposition~$2.4$]{PSY}). Their
analysis of $G$ and $\Gamma$ heavily depends upon the Classification of Finite Simple
Groups (though they state without proof that they can avoid this in the undirected case). Since our main results generalize theirs in the case of product action, we produce a
CFSG-free proof of this part of their result. 
The combination of arc-transitivity and transitivity on pairs of vertices that share an out-neighbour was also exploited to great effect by Cameron in some of his early work (see for example~\cite{Cam1,Cam2,Cam3}), although since he did not require that the pairs of vertices sharing an out-neighbour be non-adjacent, his condition is stronger than ours and his work does not directly apply to ours. It is nonetheless interesting that he did encounter the family of digraphs that we call $X_q$ in his work~\cite{Cam2}, but immediately discarded them as he was interested only in primitive actions, and did not generalise them to our family $X_q(n)$.
Our second motivation is that our definition of $\wedge$-transitivity covers the
special case where  $G$ acts transitively on each of the three sorts of pairs of vertices at
distance 2, namely $A^2_+\Gamma$, 
\begin{align*}
&\{(u,v)\in V\Gamma\times V\Gamma : u,v \textrm{ are 
non-adjacent and }
 \Gamma^+(u)\cap \Gamma^-(v)\neq \emptyset\}, \textrm{ and}\\
A^2_-\Gamma=&\{(u,v)\in V\Gamma\times V\Gamma : u,v \textrm{ are 
non-adjacent and }
 \Gamma^-(u)\cap \Gamma^-(v)\neq \emptyset\}
 \end{align*}
(which is a very natural definition of $2$-distance-transitivity on digraphs). 
There is one more remark we wish to make in this direction. (Given a digraph $\Gamma$, denote by $\Gamma^{opp}$ the digraph with $V\Gamma^{opp}=V\Gamma$ and with $A\Gamma^{opp}=\{(u,v): (v,u)\in A\Gamma\}$.) If $G$ acts transitively on $A\Gamma$ and on $\{(u,v)\in V\Gamma\times V\Gamma : u,v \textrm{ are 
non-adjacent and }
 \Gamma^-(u)\cap \Gamma^-(v)\neq \emptyset\}$, then our arguments apply immediately to $G$ and to $\Gamma^{opp}$.

We stress that in Definition~\ref{def} we assume that $T$  does not act regularly on $\Delta$; this condition is imposed in order to avoid the case that $\Gamma$ is a Cayley graph on $N$. 

%In Section~\ref{subsub}, we briefly describe the strategy for our proof of Theorem~\ref{main-thm}. 
Throughout the rest of the paper, we let $\Delta,n,\Omega,T,H$ and $W$ be as above. Furthermore, let $G$ be a product action type subgroup of $W$ and  let $\Gamma$ be a connected digraph with $\Omega=V\Gamma$ and with $G$ acting $\wedge$-transitively on $\Gamma$.

\begin{remark}\label{GpreservesHammingDist}
{\rm The action of a group $G$ of product action type always preserves the Hamming
distance between vertices.  This is easy to verify, but very important.}
\end{remark}

We fix, once and for all, $\delta$ an element of $\Delta$,
$\alpha=(\delta,\ldots,\delta)\in V\Gamma$ and
$\beta=(\delta_1,\ldots,\delta_n)\in\Gamma^-(\alpha)$.

\begin{remark}\label{action}{\rm
Since $N$ is transitive on $V\Gamma$, we have $G=NG_\alpha$ and, as $N$ acts trivially by conjugation on $\{T_1,\ldots,T_n\}$ and $G$ acts transitively on $\{T_1,\ldots,T_n\}$, we see that $G_\alpha$ acts transitively by conjugation on $\{T_1,\ldots,T_n\}$.}
\end{remark}

\subsection{Structure of the paper. }\label{subsub}Our proof is divided in various cases, depending upon the Hamming distance between $\alpha$ and $\beta$. The case that $\alpha$ and $\beta$ are at Hamming distance $1$ is studied in Section~\ref{dist=1}. In Section~\ref{morethan2}, we study the case that $\alpha$ and $\beta$ are at Hamming distance $\geq 2$. First we show that $n=2$ (in particular, $\alpha$ and $\beta$ are at Hamming distance $2$), then in Subsection~\ref{und} we study the case that $\Gamma$ is undirected, and finally in Subsection~\ref{dir} we study the case that $\Gamma$ is directed.

\section{Examples of $\wedge$-transitive digraphs with product action}\label{examples}

In this section, we explain how to construct the graphs and digraphs that are listed in Theorem~\ref{main-thm}.
First we give the definition of orbital graph. This will be required in the construction of some of the  examples that follow. 

\begin{definition}\label{orbital}{\rm Let $G$ be a transitive permutation group on the set $\Omega$ and let $\alpha$ and $\beta$ be elements of $\Omega$. The {\em orbital graph} $(\beta,\alpha)^G$ is the graph with vertex set $\Omega$ and with arc set $\{(\beta^g,\alpha^g) : g\in G\}$. The group $G$ acts transitively on the arcs of $(\beta,\alpha)^G$, the in-neighbourhood of $\alpha$ is $\beta^{G_\alpha}$ and the out-neighbourhood of $\beta$ is $\alpha^{G_\beta}$.}
\end{definition}

In the next example we describe the Hamming distance and the well-known Hamming graphs.
\begin{example}\label{example-hamming}{\rm  We
    say that  $\omega=(\delta_1,\ldots,\delta_n)$ and
    $\omega'=(\delta_1',\ldots,\delta_n')$ of $\Omega=\Delta^n$ are at
    {\em Hamming distance} $k$ if 
    $\omega$ and $\omega'$ agree in all but $k$ coordinates, that is, $k=|\{i\in
    \{1,\ldots,n\}: \delta_i\neq \delta_i'\}|$.  We denote this by $d_H(\omega,\omega')=k$.

Write $m=|\Delta|$. Let $\H(m,n)$ be the graph with vertex set
$\Omega$ and with $\omega$ 
adjacent to $\omega'$ if $d_H(\omega,\omega')=1$. The group $W=\Sym(\Delta)\wr \Sym(n)$ acts
transitively on the vertices of $\H(m,n)$, the stabilizer in $W$ of the vertex
$\alpha=(\delta,\ldots,\delta)$ of $\H(m,n)$ is
$\Sym(\Delta\setminus\{\delta\})\wr \Sym(n)$ and acts
transitively on the neighbourhood of $\alpha$ in $\H(m,n)$ and on the
vertices  at distance $2$ from $\alpha$ in
$\H(m,n)$. Therefore $W$ acts $2$-distance-transitively (equivalently, since $\H(m,n)$ a graph, $\wedge$-transitively) on
$\H(m,n)$. 
When $n=2$, $\H(m,2)$ has diameter
$2$ and so the complement of $\H(m,2)$ is also 
  $2$-distance-transitive.} 
\end{example}

The directed graphs arising in the next example show some remarkable properties which (to the best of our knowledge) have not been noticed previously. 

\begin{example}[\textsc{The directed graphs} $X_q$ and $X_q(n)$]\label{Ex2}{\rm
Let $q$ be a power of a prime with $q\equiv 3\mod 4$ and $q\geq 7$, and let
$H=\SL(2,q)$ be the special linear group.  Note that as
$q\equiv 3\mod 4$, the element $-1$ of $\mathbb{F}_q$ is not a square.  Let
$V=\mathbb{F}_q\times \mathbb{F}_q$ be the vector space of dimension
$2$ of row vectors over the field $\mathbb{F}_q$ of size $q$. Let
$\Delta$ be the set of orbits of the group of diagonal matrices
\[
C=\left\{\left(
\begin{array}{cc}
x^2&0\\
0&x^2\\
\end{array}
\right)| x\in \mathbb{F}_q, x\neq 0\right\}
\]
acting on the set of non-zero vectors $V^*=V\setminus\{(0,0)\}$. Since $C$ acts
semiregularly on $V^*$ and $|C|=(q-1)/2$, each
orbit of $C$ on $V^*$ has size $(q-1)/2$. As $|V^*|=q^2-1$, we obtain
that $\Delta$ contains $2(q+1)$ elements. For $(a,b)\in
V^*$, we denote by $[a,b]$
the
element of $\Delta$ containing $(a,b)$. Since
$-1$ is not a square, we see that $\Delta=\{[a,\pm 1],[\pm 1,0] :
a\in \mathbb{F}_q\}$. 

The only non-identity proper normal
subgroup of $H$ is the centre $Z=\langle z\rangle$ (where $z$ is the scalar matrix with entries $-1$) 
%\[z=
%left(
%\begin{array}{cc}
%-1&0\\
%0&-1\\
%\end{array}
%\right),
%\]
and the orbits of $Z$ on
$\Delta$ are $\{[a,1],[-a,-1]\}$ (for each $a\in
\mathbb{F}_q$) and $\{[1,0],[-1,0]\}$. In particular, $H$ acts
faithfully on $\Delta$. Furthermore, the action of $H$ on the $Z$-orbits of $\Delta$ is the natural $2$-transitive action of $H/Z=\PSL(2,q)$ on the $q+1$
points of the projective line. 
The stabilizer in $H$ of the
element $[1,0]$ is the
subgroup
\begin{equation}\label{eq1}
H_{[1,0]}=\left\{\left(
\begin{array}{cc}
x^2&0\\
y&x^{-2}\\
\end{array}
\right): x,y\in \mathbb{F}_q, x\neq 0\right\},
\end{equation}
which has $4$ orbits on $\Delta$, namely $\{[1,0]\}$, $\{[-1,0]\}$,
$\{[a,1]: a \in \mathbb{F}_q\}$ and $\{[a,-1]: a\in
\mathbb{F}_q\}$, of size $1,1,q$ and $q$, respectively.

We let $X_q$ be the $H$-orbital graph $([1,0],[0,1])^H$. It is a computation (using  that $-1$ is not a square) to see that there is no $h\in H$ such that
$([1,0],[0,1])^h=([0,1],[1,0])$. Therefore $X_q$ is a directed
graph of in- and out-valency $q$. For example, $X_q^+([1,0])=\{[a,1]: a\in \mathbb{F}_q\}$ and $X_q^-([0,1])=\{[1,a]: a\in \mathbb{F}_q\}$. By applying 
\begin{equation}\label{eq2}
\iota=\left(
\begin{array}{cc}
0&-1\\
1&0\\
\end{array}
\right)
\end{equation}
to the set $X_q^{-}([0,1])$, we obtain $X_q^-([1,0])=\{[a,-1]: a\in \mathbb{F}_q\}$. This gives that $VX_q=\{[1,0],[-1,0]\}\cup X_q^+([1,0])\cup X_q^-([1,0])$ and $[-1,0]$ is the unique vertex of $X_q$ not adjacent to $[1,0]$. Now vertex transitivity shows that, for each vertex $v$, there exists a unique vertex which is not
adjacent to $v$ (namely $v^z$). 

We have $(X_q^+([1,0]))^z=X_q^-([1,0])$
and similarly $(X_q^-([1,0]))^z=X_q^+([1,0])$. Therefore, for each vertex $v$, we have 
$X^+(v)=X^-(v^z)$ and $X^-(v)=X^+(v^z)$. Therefore $X^+(v)\cap
X^+(v^z)=\emptyset$ and  $A^2_+X=\emptyset$. 
Consider the matrix
\[o=\left(
\begin{array}{cc}
0&1\\
1&0
\end{array}
\right)
\]
and the map $\circ:\Delta\to \Delta$ defined by $v\mapsto v^o$. It is an easy computation to show that $\circ$ determines a graph isomorphism from $X_q$ to $X_q^{opp}$. So, $X_q\cong X_q^{opp}$.

Let $n\geq 2$, let $W=H\wr\Sym(n)$ and let
$\alpha=([1,0],[1,0],\ldots,[1,0])$ and
$\beta=([0,1],[1,0],\ldots,[1,0])$ be in $\Delta^n$. We denote by
$X_q(n)$ the orbital graph $(\beta,\alpha)^W$. Clearly, $X_q(1)=X_q$.}
\end{example}

Since $X_q \cong X_q^{opp}$, we also have $X_q(n) \cong X_q(n)^{opp}$.  Therefore when $\Gamma=X_q(n)$, $\Aut(\Gamma)$ is transitive not only on $A^2_+\Gamma$, but on $A^2_-\Gamma$.  Thus, Theorem~\ref{main-thm} in fact tells us that in the situation we are studying, our definition of $\wedge$-transitivity for digraphs actually forces transitivity on $A^2_-\Gamma$ as well. Unfortunately, the automorphism group of $X_q(n)$ is not transitive on $A^2X_q(n)$ when $n >1$, so the strongest form of 2-distance-transitivity is not forced.

In Example~\ref{Ex2}, we exclude the case that $q=3$. In fact, for $q=3$, the graphs  $X_q(n)$ are still well-defined, but, since the socle of the group $H=\SL(2,3)\cong Q_8\rtimes C_3$ acts regularly on $\Delta$, we get that $X_q(n)$ is a  Cayley graph (recall that we are not concerned with Cayley graphs in this paper).

The graphs $X_q(n)$ are $\wedge$-transitive, as the following lemma explains.% (The group $W$ and the graphs $X_q(n)$ are as in Example~\ref{Ex2}.) 
\begin{lemma}\label{exa1}
Let $W$, $X_q(n)$ be as in Example~\ref{Ex2}. Then $W$ acts $\wedge$-transitively on $X_q(n)$.
\end{lemma}
\begin{proof}
The proof consists of routine computations. For a detailed argument see~\cite{arxiv}.
\end{proof}

\section{The case $d_H(\alpha,\beta)=1$}\label{dist=1}

In this section we prove Theorem~\ref{main-thm} when $d_H(\alpha,\beta)=1$. We
start by recalling the definition of a {\em tournament}, which surprisingly is necessary in our arguments. A tournament
is a directed graph obtained by assigning a direction to each edge in
an undirected complete graph (that is, every pair of vertices is connected by a
single directed edge). A tournament is called {\em symmetric} if its
automorphism group is transitive on the
arcs. A finite symmetric tournament $\mathscr{T}$ has an odd number of
vertices, say $|V\mathscr{T}|=1+2k$, and every vertex has $k$ in-neighbours and
$k$ out-neighbours. The {\em Payley tournament} $\mathscr{T}_q$ is the
tournament with vertices the elements of the finite field
$\mathbb{F}_q$, where $q\equiv 3 \mod 4$, and with an arc from $a$ to
$b$ when $b-a$ is a non-zero square in $\mathbb{F}_q$ (that
is, $b-a=x^2$ for some $x\in \mathbb{F}_q\setminus \{0\}$).

\begin{lemma}\label{tournamentprel}Let $\mathscr{T}$ be a finite symmetric
  tournament and let $H$ be a group of automorphisms acting transitively on
  the arcs of $\mathscr{T}$. Then $\mathscr{T}\cong \mathscr{T}_q$, for some
  $q\equiv 3\mod 4$, and the socle of $H$ acts regularly on the
  vertices of $\mathscr{T}$.
\end{lemma}

\begin{proof}
Berggren~\cite{Berg} shows that if $\mathscr{T}$ is a finite symmetric
tournament, then $\mathscr{T}$ is isomorphic to $\mathscr{T}_q$ for some $q\equiv 3\mod
4$. In particular, we may assume that $\mathscr{T}=\mathscr{T}_q$.
Moreover,~\cite[Theorem~A]{Berg} gives that the automorphism group
$\mathrm{Aut}(\mathscr{T}_q)$ of $\mathscr{T}_q$ is the group of all affine permutations
of $\mathbb{F}_q$ of the form $\tau_{\sigma,x^2,c}:a\mapsto x^2a^\sigma +c$, where $c\in \mathbb{F}_q$, $x\in\mathbb{F}_q\setminus\{0\}$, and $\sigma\in\mathrm{Gal}(\mathbb{F}_q)$. Using this description of $\mathrm{Aut}(\mathscr{T}_q)$, it is easy to see that if $H$ acts transitively on
the arcs of $\mathscr{T}_q$, then $A=\{\tau_{id,x^2,c}:
x,c\in\mathbb{F}_q,x\neq 0\}$ is a subgroup of $H$ (where $id$ denotes
the identity Galois automorphism of $\mathbb{F}_q$). Now the socle of
$A$ is $\{\tau_{id,1,c}: c\in \mathbb{F}_q\}$ and coincides with
the socle of $\mathrm{Aut}(\mathscr{T}_q)$. Clearly $T=\{\tau_{id,1,c}:
c\in \mathbb{F}_q\}$ acts regularly on the vertices of $\mathscr{T}_q$.
\end{proof}

Before proceeding, we need the following definition. (The normal quotient technique is a very important idea introduced in~\cite{Praeger} which has proven useful in the investigation of many  graphs~\cite{Pr1,Pr2}.)

\begin{definition}\label{def:nq}{\rm
Let $G$ be a group acting transitively on the digraph $\Gamma$, and let $C$ be a normal subgroup of $G$. Let $\alpha^C$ denote the $C$-orbit containing $\alpha\in  V\Gamma$. The \emph{normal quotient} $\Gamma_C$ is the graph whose vertices are the $C$-orbits on $V\Gamma$, with an arc between distinct vertices $\alpha^C$ and $\beta^C$ if
and only if there is an arc of $\Gamma$ between $\alpha'$ and $\beta'$, for some $\alpha' \in \alpha^C$ and $\beta'\in \beta^C$.
}
\end{definition} 

The following proposition is the most substantial result of this section. The proof is quite long and involved, however, it is elementary  and  we do not make use of the Classification of the Finite Simple Groups.  The corollary that follows it will complete the proof of Theorem~\ref{main-thm} in the case where $d_H(\alpha,\beta)=1$.
\begin{proposition}\label{propdim1}
Let $X$ be a connected $H$-orbital graph and let $\delta'$ be an arbitrary vertex of
$X$ (so $\delta' \in \Delta$). Assume that any two vertices of $X^-(\delta')$ are adjacent. Then
either $X$ is the complete graph, or $X=X_q$, for some $q\equiv 3\mod
4$.    
\end{proposition}

\begin{proof}
Fix $\delta_0$ a vertex of $X$. Suppose that $X$ is undirected. As $X^-(\delta_0)=X(\delta_0)$ is a complete graph and as every vertex of $X(\delta_0)$ is adjacent to $\delta_0$, we obtain that $\{\delta_0\}\cup X(\delta_0)$ is a connected component of $X$. Since $X$ is connected, we see that $X$ is complete. Therefore, in the rest of the proof we may assume that $X$ is a digraph. Let $q$ be the in-valency of $X$.

Suppose that $|VX|=1+2q$, that is, $VX=\{\delta_0\}\cup 
X^+(\delta_0)\cup X^-(\delta_0)$. Since $H$ acts transitively on the
vertices of $X$, we see that any two vertices of $X$ are adjacent. In
particular, $X$ is a tournament. Since we are assuming that $T$ does not act regularly on
$VX$, from Lemma~\ref{tournamentprel} we obtain a contradiction.
In particular, in
the rest of the proof we may assume that $|VX|>1+2q$. If $q=1$, then
$X$ is a directed cycle and its automorphism group is a cyclic group. 
 The socle of $\mathrm{Aut}(X)$ acts regularly on $VX$ which
 again contradicts our hypothesis on $T$. Thus $q>1$.

As $H_{\delta_0}$ acts transitively on $X^-(\delta_0)$ and as any two
vertices of $X^-(\delta_0)$ are adjacent, we see that the induced subgraph
of $X$ on $X^-(\delta_0)$ is a symmetric tournament. In
particular, $q$ is odd. 
Now we prove eight claims from which the result will follow.

\smallskip
\noindent\textsc{Claim~1. }The induced subgraph of $X$ on
$X^+(\delta_0)$ is a symmetric tournament.
\smallskip

\noindent Let $\delta'$ be in $X^+(\delta_0)$, let $\delta''$ be in $X^-(\delta_0)$ and let $Y=X^+(\delta'')\cap X^-(\delta_0)$ be the out-neighbours of $\delta''$ in
$X^-(\delta_0)$. Since the induced subgraph  of $X$ on $X^-(\delta_0)$ is a
tournament, we have $|Y|=(q-1)/2$. As $(\delta'',\delta_0)$ and
$(\delta_0,\delta')$ are arcs of $X$ and $H$ is transitive on $AX$, there exists $h\in H$ with
$({\delta''},\delta_0)^h=(\delta_0,\delta')$. We obtain $Y^h=X^+(\delta_0)\cap X^-(\delta')$ and $\delta'$
 has $|Y^h|=(q-1)/2$ in-neighbours in $X^+(\delta_0)$. Since $H_{\delta_0}$
 is transitive on $X^+(\delta_0)$, the induced subgraph  of $X$ on
 $X^+(\delta_0)$ has out-valency $(q-1)/2$ and in-valency $(q-1)/2$ and
 hence it is a symmetric tournament.~$_\blacksquare$

\smallskip
\noindent\textsc{Claim~2. }Let $\delta'$ be in $X^+(\delta_0)$. Then
$\delta'$ has exactly $(q-1)/2$ in-neighbours in $X^+(\delta_0)$ and
exactly $(q-1)/2$ in-neighbours in $X^-(\delta_0)$.
\smallskip

\noindent As $X^+(\delta_0)$ is a symmetric tournament, $\delta'$ has $(q-1)/2$
in-neighbours in $X^+(\delta_0)$, that is $|X^-(\delta')\cap
X^+(\delta_0)|=(q-1)/2$. Set $Y=X^-(\delta')\setminus(\{\delta_0\}\cup
X^+(\delta_0))$ and let $\delta''$ be in $Y$. As $X^-(\delta')$ is a
symmetric tournament and as $\delta_0,\delta''\in X^-(\delta')$, we have
that $\delta_0$ and $\delta''$ are adjacent. Since $\delta''\notin
X^+(\delta_0)$, we get $\delta''\in X^-(\delta_0)$. As $\delta''$ is an
arbitrary element of $Y$, we have $Y\subseteq X^-(\delta_0)$ and
$\delta'$ has $|Y|=(q-1)/2$ in-neighbours in $X^-(\delta_0)$.~$_\blacksquare$

\smallskip

Claim~2 shows that for each $\delta'\in X^+(\delta_0)$, we
have $X^-(\delta')\subseteq (\{\delta_0\}\cup X^+(\delta_0)\cup
X^-(\delta_0))$. If, for every $\delta'\in X^+(\delta_0)$, we also have  $X^+(\delta')\subseteq (\{\delta_0\}\cup X^+(\delta_0)\cup
X^-(\delta_0))$, then (using the transitivity of $H$ on $VX$ together with a
connectedness argument) we obtain $VX=\{\delta_0\}\cup X^+(\delta_0)\cup X^-(\delta_0)$, which contradicts
$|VX|>1+2q$. This shows that there exists $\delta'\in X^+(\delta_0)$ and
$\delta_0^*\in X^+(\delta')$ with $\delta_0^*\notin (\{\delta_0\}\cup X^+(\delta_0)\cup X^-(\delta_0))$.

\smallskip
\noindent\textsc{Claim~3. } $X^-(\delta_0^*)=X^+(\delta_0)$.
\smallskip

\noindent Let $\delta''$ be an out-neighbour of $\delta'$ in
$X^+(\delta_0)$. Since $\delta_0^*,\delta''\in X^+(\delta')$, by Claim~1
and by vertex transitivity, we obtain that $\delta_0^*$
and $\delta''$ are adjacent. If $\delta_0^*\in X^-(\delta'')$, then by
Claim~2 applied to $\delta''$, we see that $\delta_0^*\in
(\{\delta_0\}\cup X^+(\delta_0)\cup X^-(\delta_0))$, which
contradicts our choice of $\delta_0^*$. Therefore $\delta_0^*\in X^+(\delta'')$. Since $\delta''$
is an arbitrary out-neighbour of $\delta'$ in $X^+(\delta_0)$, we see
that every out-neighbour of $\delta'$ in $X^+(\delta_0)$ is an
in-neighbour of $\delta_0^*$. The vertex $\delta'$ was an arbitrary element of $X^+(\delta_0)$
in this argument, so since the induced subgraph of $X$ on $X^+(\delta_0)$ is a symmetric tournament, every vertex of $X^+(\delta_0)$ is an in-neighbour of some element of $X^+(\delta_0)\cap X^-(\delta_0^*)$.
Hence $X^+(\delta_0)\subseteq
X^-(\delta_0^*)$. Since $q=|X^+(\delta_0)|=|X^-(\delta_0^*)|$, we have $X^+(\delta_0)=X^-(\delta_0^*)$.~$_\blacksquare$
 
\smallskip
\noindent\textsc{Claim~4. }$X^+(\delta_0^*)=X^-(\delta_0)$.
\smallskip

\noindent We first show that $\delta_0^*$ has at least one out-neighbour in $X^-(\delta_0)$. Fix an element $w$ in $X^+(\delta_0)$ and write
$U=X^-(w)\cap X^+(\delta_0)$ and $V=X^-(w)\cap
X^-(\delta_0)$. From Claim~2, $|U|=|V|=(q-1)/2$. As $X^-(\delta_0)$ is a
symmetric tournament, by $H$-transitivity, we have that
$X^-(w)=\{\delta_0\}\cup U\cup V$ is a symmetric tournament. Let
$\delta''$ be in $U$. Now, as $\delta''$ has $(q-1)/2$ out-neighbours
in $X^-(w)$ and as $\delta_0$ is not an out-neighbour of
$\delta''$ (because $\delta''\in U\subseteq
X^+(\delta_0)$), we obtain by the pigeon-hole principle that $\delta''$ has an out-neighbour  $\delta_0'$ in
$V$, that is, $\delta_0'\in X^+(\delta'') \cap V$. As $\delta''\in X^+(\delta_0)=X^-(\delta_0^*)$, we see that
$\delta_0^*\in X^+(\delta'')$. Therefore, $\delta_0'$ and $\delta_0^*$ are both in $X^+(\delta'')$.  From Claim~1 applied to $\delta''$ we get that
$\delta_0'$ and $\delta_0^*$ are adjacent. Since $\delta_0'\in V\subseteq
X^-(\delta_0)$ and since $X^-(\delta_0^*)=X^+(\delta_0)$, we have
$\delta_0'\notin X^-(\delta_0^*)$. Therefore we must have that
$\delta_0'\in X^+(\delta_0^*)$.

Now that we have shown that $\delta_0^*$ has one out-neighbour $\delta_0'$ in
$X^-(\delta_0)$, using an argument similar to the argument in the proof of
Claim~3, we obtain that every element of $X^-(\delta_0)$ is an
out-neighbour of $\delta_0^*$. So $X^+(\delta_0^*)=X^-(\delta_0)$.~$_\blacksquare$

\smallskip
\noindent\textsc{Claim~5. }$VX=\{\delta_0,\delta_0^*\}\cup X^+(\delta_0)\cup
X^-(\delta_0)$ and $|VX|=2(1+q)$.
\smallskip

\noindent We show that every vertex $v$ in the induced subgraph of $X$ on $V=\{\delta_0,\delta_0^*\}\cup X^+(\delta_0)\cup
X^-(\delta_0)$ has in-valency $q$, from which the claim follows by
connectedness. If $v=\delta_0$, then there is nothing to prove. If 
$v=\delta_0^*$, then from Claim~3 we have $X^-(\delta_0^*)=X^+(\delta_0)\subseteq
V$. Also, if $v\in X^+(\delta_0)$, then from Claim~2 we have $X^-(v)\subseteq
V$. It remains to consider $v\in X^-(\delta_0)$. Applying the argument in
Claim~2 with $v=\delta'$ and with $\delta_0^*=\delta_0$, we obtain that
$v$ has $(q-1)/2$ in-neighbours in $X^+(\delta_0^*)$ and $(q-1)/2$
in-neighbours in $X^-(\delta_0^*)$. As $\delta_0^*$ is also an
in-neighbour of $v$, we obtain $X^-(v)\subseteq V$ from Claims~3
and~4.~$_\blacksquare$

\smallskip

Since $H$ acts transitively on the vertices of $X$, we have that for
every vertex $v$, there exists a {\em unique} vertex $v^*$ with
$X^+(v)=X^-(v^*)$ and $X^-(v)=X^+(v^*)$. In particular, the set
$\mathscr{B}=\{\{v,v^*\}: v\in VX\}$ is a system of imprimitivity
for the action of $H$ on $VX$. Let $C$ be the kernel of the action of
$H$ on $\mathscr{B}$. 

\smallskip
\noindent\textsc{Claim~6. }Let $v$ be in $VX$ and let $v'$ be in
$VX\setminus\{v,v^*\}$. Then the induced subgraph  of $X$ on
$\{v,v^*,v',(v')^*\}$ is a directed cycle.
\smallskip

\noindent From Claim~5, we see that $v$ and $v'$ are adjacent, so
replacing $(v,v')$ by $(v',v)$ if necessary, we may assume that $v'\in
X^+(v)$. 
As $v'\in X^+(v)=X^-(v^*)$, $(v',v^*)$ is an arc of $X$. Since
$(v')^*\neq v^*$ and since $v$ is adjacent to every element different
from $v^*$, we obtain that either $(v')^*\in X^+(v)$ or $(v')^*\in
X^-(v)$. If $(v')^*\in X^+(v)$, then $v',(v')^*\in X^+(v)$ and so from
Claim~1, $v'$ and $(v')^*$ are adjacent, a contradiction. Therefore
$((v')^*,v)$ is an arc of $X$. As $(v')^*\in X^-(v)=X^+(v^*)$, $(v^*,(v')^*)$
is also an arc of $X$.~$_\blacksquare$

\smallskip

\noindent\textsc{Claim~7. }$H$ contains a unique element of order $2$
and $|C|=2$.
\smallskip

\noindent As $|VX|=2(q+1)$ is even and $H$ acts transitively on $X$,
the group $H$ contains an element $h$ of order $2$. Assume that $h\in
H\setminus C$. As $h\notin C$, there exists
$v\in VX$ with $\{v,v^*\}^h\neq \{v,v^*\}$. Set $v'=v^h$. From
Claim~6, the induced
subgraph  of $X$ on $\{v,v^*,v',(v')^*\}$ is a directed cycle
which is $h$-invariant because $h^2=1$. As the automorphism group of a directed
cycle of length four is a cyclic group whose generator squares to an involution mapping $v$ to
$v^*$, we obtain $v'=v^h=v^*$, a contradiction. Therefore, every
involution of $H$ lies in $C$.

Since the blocks of $\mathscr{B}$ have size $2$, we have that $C$ is
an elementary abelian $2$-group. Let $h$ be an element of
$C\setminus\{1\}$ and assume that $h$ fixes a vertex, $v$ say, of
$X$. Let $v'$ be any vertex of $X$ with $v'\not\in\{ v,v^*\}$. Now, from Claim~6, the
induced subgraph of $X$ on $\{v,v^*,v',(v')^*\}$ is a directed cycle
which is $h$-invariant. As the automorphism group of a directed cycle
is a cyclic group acting regularly and as $h$ fixes $v$, we obtain
that $h$ fixes $v,v^*,v',(v')^*$. Since $v'$ is an arbitrary element of
$VX$ with $v'\not\in\{ v,v^*\}$, we obtain $h=1$. This shows that $|C|=2$.~$_\blacksquare$

\smallskip

We let $z$ denote the generator of $C$, and $T$ the socle of $H$. We have $v^z=v^*$, for each $v\in VX$.

As $T$ acts
transitively on $VX$ and since $|VX|$ is even, $C\leq T$. 
Let $S$ be a Sylow $2$-subgroup of $T$. Since $H$ has a unique
involution, the group $S$ has a unique involution (namely $z$). It
follows from~\cite[Proposition~(4.4), p.~59]{Suz} that $S$ is either a cyclic group of order $2^a$
(for $a\geq 1$), or a quaternion
group of order $2^a$ (for $a\geq 3$). It follows that either  $S/C$ is a cyclic group of order $2^{a-1}$,
or a dihedral group of order $2^{a-1}$ (if $a\geq 4$), or an
elementary abelian $2$-group of order $4$ (if $a=3$).

Since the group $C$ acts transitively on $\{v,v^*\}$ (for each $v\in
VX$), the system of imprimitivity $\mathscr{B}$ consists of
the orbits of $C$ on $VX$. In particular, the quotient graph $X_C$ is a normal
quotient. From Claim~5, $X_C$ is an undirected complete graph with
$q+1$ vertices. Write $\overline{H}=H/C$. Since $H$ acts
arc-transitively on $X$, the group $\overline{H}$ acts
arc-transitively on $X_C$. In particular, $\overline{H}$ is a
$2$-transitive group of degree $q+1$. 
Let $\overline{T}=T/C$ be the socle of
$\overline{H}$. By a celebrated theorem of Burnside~\cite[Theorem~4.1B]{DM}, $\overline{T}$
is either a regular elementary abelian $p$-group (for some prime $p$),
or a non-regular non-abelian simple group. Assume that $\overline{T}$ is
abelian. Since $|\overline{T}|=q+1$, we have $|T|=2(q+1)$ and $T$
acts regularly on $VX$, a contradiction. 

This shows that $\overline{T}$ is a
non-regular non-abelian simple group whose Sylow $2$-subgroup $S/C$ is
either cyclic, or dihedral  or elementary abelian of order $4$.  

\smallskip
\noindent\textsc{Claim~8. } $T=\mathrm{SL}(2,r)$ for some odd $r$. 
\smallskip

\noindent From~\cite[Corollary~2, p.~144]{Suz}, we see that the Sylow
$2$-subgroup of a simple group is not cyclic. 
If $C$ splits over $T$ (that is, $C$ has a complement, $L$ say, in
$T$), then $T=L\times C$ for some finite non-abelian simple group $L$. As $L$ has even order by the Odd order theorem, the group $T$ has more than one involution, which contradicts Claim~$7$. Thus $C$ does not split over $T$. Therefore $T$ is a quotient of the universal covering group $U$ of $\overline{T}$, that
is, $T\cong U/Z$ for some central subgroup $Z$ of $U$. We now show that $\overline{T}\cong \PSL(2,r)$, for some odd $r$.

Suppose that $S/C$ is a
dihedral group. 
From the classification of Gorenstein and
Walter~\cite{GW} of the non-abelian simple groups with a dihedral
Sylow $2$-subgroup, we see that either $\overline{T}\cong \Alt(7)$ or
$\overline{T}\cong \mathrm{PSL}(2,r)$ for some odd $r$ with $r\equiv 1\mod 8$
or $r\equiv 7\mod 8$. From~\cite{ATLAS}, we see that $\Alt(7)$ has only
two $2$-transitive permutation representations, one of degree $7$ and
one of degree $15$. As $q+1$ is even, we obtain that
$\overline{T}\not\cong \Alt(7)$. 

We may now assume that $S/C$ is an elementary abelian group of order $4$. From the classification of
Walter~\cite{W} of the non-abelian simple groups with an abelian
Sylow $2$-subgroup, we see that either $\overline{T}\cong
\mathrm{PSL}(2,r)$ (for $r=2^b$, or for some odd $r$ with  $r\equiv 3\mod 8$  or
$r\equiv 5\mod 8$), or $\overline{T}=J_1$, or
$\overline{T}={^2}G_2(3^\ell)$ (for some odd
$\ell>1$). From~\cite[Table~5 and p.~36]{ATLAS}, we see
       that the universal covering groups of ${^2}G_2(3^\ell)$ and of
       $J_1$ are simple. Hence $T=U=\overline{T}$, a contradiction. 
Since a Sylow $2$-subgroup of $\mathrm{PSL}(2,2^b)$ has
order $2^b$, we obtain that $b=2$. Moreover, since
$\mathrm{PSL}(2,4)\cong \mathrm{PSL}(2,5)$, we can include this case
in the odd characteristic. 

From~\cite[Table~5]{ATLAS}, the universal
covering group of $\mathrm{PSL}(2,r)$, with $r$ odd, is
$\mathrm{SL}(2,r)$, whose centre has order $2$. Since
$|T|=2|\overline{T}|$, we get $T=U=\mathrm{SL}(2,r)$ and $C$ is the centre of $T$.~$_\blacksquare$

\smallskip

We are now ready to conclude the proof. 
From~\cite[Table~$7.4$, p.~197]{Peter}, the group $\mathrm{PSL}(2,r)$
has only one $2$-transitive permutation representation of even degree,
namely the natural action of degree $r+1$ on the points of the
projective line. In particular, $r=q$. Moreover, there exists $\{v,v^*\}\in \mathcal{B}$ such that the stabilizer $T_{\{v,v^*\}}$ is the Borel subgroup

\[
\left\{
\left(
\begin{array}{cc}
x&0\\
b&x^{-1}
\end{array}
\right)
: x,b\in\mathbb{F}_q,x\neq 0
\right\}.
\]
Since $|\{v,v^*\}|=2$, the stabilizer $T_v$ has index $2$ in $T_{\{v,v^*\}}$. As $T_{\{v,v^*\}}\cong
\mathbb{F}_q\rtimes \mathbb{F}_q^*$, we see that $T_{\{v,v^*\}}$ has a unique
subgroup of index $2$ and hence

\[
T_v=\left\{
\left(
\begin{array}{cc}
x^2&0\\
b&x^{-2}
\end{array}
\right)
: x,b\in\mathbb{F}_q,x\neq 0
\right\}.
\]
Therefore the action of $T$ on $VX$ can be identified
with the action of $T$ on the right cosets of $T_v$.

If $q\equiv 1\mod 4$, then the centre $C$ of $T$ is contained in $T_v$ and the action of $T$ on $VX$ is unfaithful, a contradiction.  
So $q\equiv 3\mod 4$. Moreover, $T_v$ has four orbits on $T/T_v$ of size $1,1$, $q$ and $q$ respectively. Therefore, $T$ gives rise to only two orbital graphs of out-valency $q$, namely $X_q$ and $X_q^{opp}$. As $X_q\cong X_q^{opp}$, the proof is completed.
\end{proof}
As a consequence of Propositon~\ref{propdim1} we obtain the following corollary.

\begin{corollary}\label{Hamming-lemma}
If $d_H(\alpha,\beta)=1$, then either $\Gamma=\H(m,n)$ or $\Gamma=X_q(n)$.
\end{corollary}

\begin{proof}
Since $d_H(\alpha,\beta)=1$ and since $G_\alpha$ acts transitively on $\{T_1,\ldots,T_n\}$, replacing $\beta$ by a suitable conjugate under $G_\alpha$ if necessary, we may assume that  $\beta=(\delta',\delta,\ldots,\delta)$, for some $\delta'\in \Delta\setminus\{\delta\}$. In particular,

\begin{eqnarray}\label{eq9}\nonumber
\Gamma^-(\alpha)&=&\beta^{G_\alpha}=\left(\delta'^{H_\delta}\times
\{\delta\}\times \cdots\times \{\delta\}\right)\cup \left(\{\delta\}\times \delta'^{H_\delta}\times
\{\delta\}\times \cdots\times \{\delta\}\right)\\
&&\cup\cdots\cup\left(\{\delta\}\times
\cdots\times \{\delta\}\times\delta'^{H_\delta}\right).
\end{eqnarray}
We denote by $X_1,\ldots,X_n$ the $n$ sets on the right hand side of $\beta^{G_\alpha}$ (for instance $X_1=\delta'^{H_\delta}\times \{\delta\}\times \cdots\times \{\delta\}$). For each $i\in \{1,\ldots,n\}$, any two distinct vertices in $X_i$ are at Hamming distance $1$.  Furthermore, for each  $i,j\in \{1,\ldots,n\}$ with $i\neq j$, and for each $x\in X_i,y\in X_j$, we have that $d_H(x,y)=2$. As $\alpha$ and $\beta$ are adjacent and $d_H(\alpha,\beta)=1$ and as $G$ is transitive on $A^2_+\Gamma$, this shows that if $(u,v)\in A^2_+\Gamma$, then $d_H(u,v)=2$. Furthermore, for each $i\in \{1,\ldots,n\}$, every two vertices $u,v$ of $X_i$ are adjacent because $d_H(u,v)=1$ and $u,v\in \Gamma^{-}(\alpha)$.

Let $X$ be the $H$-orbital graph $(\delta',\delta)^H$ and  let $X'$ be the connected component of $X$ containing $\delta$. Observe that $(\delta,\ldots,\delta,\nu,\delta,\ldots,\delta)\in \Gamma^-(\alpha)$ if and only if $(\nu,\delta)\in AX$.
Set $Y=X'\times \cdots \times X'$ (seen as a subset of $V\Gamma=\Delta^n$). Let $u=(\varepsilon_1,\ldots,\varepsilon_n)$ be an element of $Y$. 
With a computation similar to the case of the vertex $\alpha$, we have 
\begin{eqnarray*}
\Gamma^-(u)&=&\{(\nu,\varepsilon_2,\varepsilon_3,\ldots,\varepsilon_n): (\nu,\varepsilon_1)\in AX\}\cup \{(\varepsilon_1,\nu,\varepsilon_3,\ldots,\varepsilon_n): (\nu,\varepsilon_2)\in AX\}\\
&&\cup\cdots \cup \{(\varepsilon_1,\varepsilon_2,\ldots,\varepsilon_{n-1},\nu): (\nu,\varepsilon_n)\in AX\}.
\end{eqnarray*} For each $i\in \{1,\ldots,n\}$ and $\nu\in VX$ with $(\nu,\varepsilon_i)\in AX$, we have $\nu\in X'$ because $\varepsilon_i\in X'$ and $X'$ is a connected component of $X$. Therefore we obtain $\Gamma^-(u)\subseteq Y$. As $u$ is an arbitrary vertex in $Y$ and $\Gamma$ is connected, we have $Y=\Delta^n$. Hence $VX'=\Delta$, that is, $X$ is connected.

Let $x$ and $y$ be in $X^-(\delta)=\delta'^{H_\delta}$. From~$(\ref{eq9})$, the vertices $\gamma_x=(x,\delta,\ldots,\delta)$ and $\gamma_y=(y,\delta,\ldots,\delta)$ are in $X_1$ and hence are adjacent in $\Gamma$. In particular, it can be shown that any $g \in G_\alpha$ that takes $\gamma_x$ to $\gamma_y$ must fix the first coordinate, and hence that $x$ and $y$ must be adjacent in $X$. Since $x$ and $y$ are arbitrary elements of $X^-(\delta)$, we have that any two vertices in $X^-(\delta)$ are adjacent.

As $X$ is connected, from Proposition~\ref{propdim1} we have that either $X$ is complete (and so $\Gamma=\H(m,n)$) or $X=X_q$ (and so $\Gamma=X_q(n)$).
\end{proof}

\section{The case $d_H(\alpha,\beta)\ge 2$}\label{morethan2}

In this section, we start our analysis of the case in which $d_H(\alpha,\beta)\ge 2$, by showing that if this occurs, then in fact $n=2$ (so $d_H(\alpha,\beta)=2$) as well as some other restrictions.

\begin{lemma}\label{n=2}
Assume that $d_H(\alpha,\beta)=k\geq 2$. If $(u,v)\in A^2_+\Gamma$, then $d_H(u,v)=1$. Furthermore, $n=2$ and neither $\delta_1$ nor $\delta_2$ is fixed   by $T_\delta$.
\end{lemma}

\begin{proof}
Note that since $d_H(\alpha,\beta)=k \ge 2$, $\beta \in \Gamma^-(\alpha)$, and $G$ is arc-transitive, Remark \ref{GpreservesHammingDist} implies that any pair of adjacent vertices must be at Hamming distance $k \ge 2$.
Suppose that  $N_\alpha$ fixes $\Gamma^-(\alpha)$ point-wise. Since $N\unlhd G$, we have that $N_\gamma$ fixes $\Gamma^-(\gamma)$ point-wise for each vertex $\gamma$, and so by connectedness, $N_\alpha=1$ and $N$ acts regularly on the vertices of $\Gamma$. As we are assuming that $N$ is not regular, we have a contradiction. Therefore $N_\alpha$ does not fix $\Gamma^-(\alpha)$ point-wise. Since $N_\alpha\unlhd G_\alpha$, $\beta\in \Gamma^-(\alpha)$, and $G_{\alpha}$ is transitive on $\Gamma^-(\alpha)$, the group $N_\alpha=T_\delta^n$ does not fix $\beta$ and hence there exists a coordinate $\delta_i$ of $\beta$ with $T_\delta$ not fixing $\delta_i$, for some $i\in \{1,\ldots,n\}$. Using Remark~\ref{action} we see that, replacing $\beta$ by a suitable conjugate under $G_\alpha$ if necessary, we may assume that $i=1$. Let $t$ be in $T_\delta\setminus T_{\delta_1}$, that is, $\delta^t=\delta$ and $\delta_1^t\neq \delta_1$. Now, $g=(t,1,\ldots,1)\in
N_\alpha\subseteq G_\alpha$ and $\gamma=\beta^g\in
\Gamma^-(\alpha)$. Since 
$d_H(\beta,\gamma)=1$ but any pair of adjacent vertices is at Hamming distance $k\ge 2$, we see that $\beta$ and $\gamma$ are not adjacent in $\Gamma$. Therefore since $\beta, \gamma \in \Gamma^-(\alpha)$, we have that $(\beta,\gamma)\in A^2_+\Gamma$. Since $G$ acts transitively on  $A^2_+\Gamma$, Remark \ref{GpreservesHammingDist} implies that all the pairs in $A^2_+\Gamma$ are at Hamming distance $1$, which proves the first part of this lemma.

Suppose that $\beta$ has only one entry not fixed by $T_\delta$.  Using Remark~\ref{action} we see that, replacing $\beta$ by a suitable conjugate if necessary, we may assume that $\delta_1$ (the first coordinate of $\beta$) is not fixed by $T_\delta$ and $\delta_2,\ldots,\delta_n$
 are point-wise fixed by $T_\delta$. Therefore $T_\delta=T_{\delta_i}$ for each $2\leq i\leq n$ and 
\begin{equation}\label{eqNb}
N_\beta=T_{\delta_1}\times T_\delta\times \cdots\times T_\delta\quad\textrm{with}\quad T_\delta\neq T_{\delta_1}.
\end{equation}
Since $N$ is transitive on the vertices of $\Gamma$ and since $N_\alpha\unlhd G_\alpha$, we have that for every vertex $\gamma$ of $\Gamma$ and for every $\nu\in \Gamma^-(\gamma)$, the group $N_\gamma$ acts non-trivially only on one coordinate of $\nu$. Since $G_\alpha$ acts transitively on $\{T_1,\ldots,T_n\}$, there exists $x=\tau(h_1,\ldots,h_n)$ in $G_\alpha\subseteq W_\alpha=H_\delta\wr K$ with $1^\tau=2$. Consider the vertex 
\begin{eqnarray}\label{eqgamma}\gamma=\beta^x&=&(
\delta_{1^{\tau^{-1}}}^{h_1},
\delta_{2^{\tau^{-1}}}^{h_2},
\delta_{3^{\tau^{-1}}}^{h_3},
\ldots,
\delta_{n^{\tau^{-1}}}^{h_n})
=(
\delta_{1^{\tau^{-1}}}^{h_1},
\delta_{1}^{h_2},
\delta_{3^{\tau^{-1}}}^{h_3},
\ldots,
\delta_{n^{\tau^{-1}}}^{h_n})\in \Gamma^-(\alpha)
.
\end{eqnarray}
Since $x \in G_{\alpha}$, we have $\gamma \in \Gamma^-(\alpha)$, so $T_{\delta}$ acts non-trivially on only one coordinate of $\gamma$.
Since $T_\delta$ does not fix $\delta_1$, $T_\delta\unlhd H_\delta$ and $h_2\in H_\delta$, we obtain that  $T_\delta$ does not fix  $\delta_1^{h_2}$. As $T_\delta$ fixes $\delta_2$, we have $\delta_2\neq \delta_1^{h_2}$. Moreover, as $T_\delta$ fixes $\delta_{1^{\tau^{-1}}}^{h_1}$, we have $\delta_1\neq \delta_{1^{\tau^{-1}}}^{h_1}$. So, $d_H(\beta,\gamma)\ge 2$ (the first two coordinates of $\beta$ and $\gamma$ are distinct). Since $\beta,\gamma\in \Gamma^-(\alpha)$, we obtain from the previous paragraph that $(\beta,\gamma)\not\in A^2_+\Gamma$. So $\beta$ and $\gamma$ are adjacent in $\Gamma$, that is, either $\gamma\in \Gamma^-(\beta)$ or $\beta\in \Gamma^-(\gamma)$. Suppose that $\gamma\in \Gamma^-(\beta)$. Thus $N_\beta$ acts non-trivially on only one coordinate of $\gamma$.
Now, as $T_\delta$ does not fix $\delta_1^{h_2}$, we obtain from~\eqref{eqNb} that the second coordinate of $\gamma$ is the only coordinate not fixed by $N_\beta$.  Therefore, from~\eqref{eqNb} and~\eqref{eqgamma}, we have 
\begin{equation}\label{eqagain}
T_{\delta_1}=T_{\delta_{1^{\tau^{-1}}}^{h_1}},\; T_{\delta}\neq T_{{\delta_1}^{h_2}}\;\textrm{ and } T_{\delta}=T_{\delta_{i^{\tau^{-1}}}^{h_i}}\;\textrm{ for }i\in \{3,\ldots,n\}.
\end{equation} Since $1^{\tau^{-1}}\neq 1$, we see from~\eqref{eqNb} that $T_\delta$ fixes $\delta_{1^{\tau^{-1}}}$ and, since $T_\delta\lhd H_\delta$ and $h_1\in H_\delta$, we have 
\[T_{\delta_1}=T_{\delta_{1^{\tau^{-1}}}^{h_1}}=(T_{\delta_1^{\tau^{-1}}})^{h_1}=T_\delta^{h_1}=T_\delta.\]
Thus $T_\delta$ fixes $\delta_1$, contradicting~\eqref{eqNb}. A similar contradiction (along these lines) is obtained by supposing that $\beta\in \Gamma^-(\gamma)$. Therefore $\beta$ must have at least two entries not fixed by $T_\delta$. Now, to conclude the proof, it suffices to show that $n=2$.

Replacing $G$ by a suitable conjugate in $H\wr\Sym(n)$ if necessary, we may assume that the first two coordinates of $\beta$ (that is, $\delta_1$ and $\delta_2$) are not fixed by $T_\delta$.  Let $t'\in T_\delta\setminus
T_{\delta_2}$ so $\delta_2^{t'}\neq \delta_2$, and set
$n'=(t,t',1,\ldots,1)\in N_\alpha$. As $\gamma'=\beta^{n'}$ has exactly two entries different from $\beta$ (namely in the first two coordinates), we see
that $(\beta,\gamma')\notin A^2_+\Gamma$ (the pairs in $A^2_+\Gamma$ are at Hamming distance $1$). Since $\beta,\gamma'\in \Gamma^-(\alpha)$, we see that $\beta$ and $\gamma'$ are adjacent and hence by Remark \ref{GpreservesHammingDist} since $G$ acts arc-transitively, $k=2$. 
In particular, as $\beta$ is adjacent to $\alpha$, we have $\beta=(\delta_1,\delta_2,\delta,\ldots,\delta)$.

Suppose that $n\geq 3$. Since $G_\alpha$ acts transitively on $\{1,\ldots,n\}$, there exists $x=\tau(h_1,\ldots,h_n)\in G_\alpha\leq W_\alpha=H_\delta\wr K$ with $3^\tau=1$. Since the third coordinate of $\beta$ is $\delta$, we obtain that the first coordinate of  $\gamma=\beta^x$ is $\delta$. Since $\beta$ and $\gamma$ each have $n-2$ coordinates equal to $\delta$ and since the first coordinate of $\beta$ is $\delta_1\neq \delta$, we obtain that $d_H(\beta,\gamma)\ge 2$. Therefore $(\beta,\gamma)\not\in A^2_+\Gamma$. As $\beta,\gamma\in \Gamma^-(\alpha)$, the vertices $\beta$ and $\gamma$ are adjacent in $\Gamma$ and in particular $d_H(\beta,\gamma)=k=2$. Since $\beta$ and $\gamma$ differ in the first coordinate and in the $i$th coordinate for some  $i\in \{3,\ldots,n\}$ such that the $i$th coordinate of $\gamma$ is not $\delta$, we obtain that $\gamma=(\delta,\delta_2,\ldots)$, that is, the second coordinate of $\gamma$ equals the second coordinate  of $\beta$. Set $n''=(1,t',1,\ldots,1)\in N_\alpha$. The vertex $\gamma^{n''}$ is adjacent to $\alpha$ and $d_H(\gamma^{n''},\beta)=3$, thus $(\beta,\gamma'')\not\in A^2_+\Gamma$ and $\beta$ and $\gamma''$ are adjacent, contradicting that $k=2$. This yields $n=2$.
\end{proof}

%When $G$ acts primitively on $V\Gamma$, the proof of Lemma~\ref{n=2} is much simplier. Indeed, $H$ is primitive on $\Delta$ and hence $\delta$ is the unique element of $\Delta$ fixed by $T_\delta$. Therefore the detailed analysis on the fixed points of $T_\delta$ in the proof of Lemma~\ref{n=2} becomes irrelevant and the argument is greatly simplified.

%\section{$\alpha$ and $\beta$ are at Hamming distance $2$}\label{}
%We continue our analysis by considering the case that $\alpha$ and $\beta$ are at Hamming distance exactly $2$. This will conclude the proof of Theorem~\ref{main-thm}. 
From Lemma~\ref{n=2}, we have $n=2$, $\alpha=(\delta,\delta)$, $\beta=(\delta_1,\delta_2)$ and $\delta_1,\delta_2$ are not fixed by $T_\delta$.
We start our analysis with a rather technical lemma.  This tells us that if two rows (or columns) of 
$\Delta^2$ are in the same $H_\delta$-orbit, then there is an element of $G$ that fixes $\alpha$ and takes the first row (or column) to the second (without exchanging the coordinates).
\begin{lemma}\label{Psfavourite}
For $\varepsilon_1,\varepsilon_2\in\Delta$ and for $i\in \{1,2\}$, we have $\pi_i(G_i\cap G_{(\varepsilon_1,\varepsilon_2)})=H_{\varepsilon_i}$.
\end{lemma}

\begin{proof}
Fix $i\in \{1,2\}$ and write $R=\pi_i(G_i\cap G_{(\varepsilon_1,\varepsilon_2)})$. Since $G_{(\varepsilon_1,\varepsilon_2)}\leq W_{(\varepsilon_1,\varepsilon_2)}$, we see that $R\leq H_{\varepsilon_i}$. Also, as $T_{\varepsilon_1}\times T_{\varepsilon_2}\leq
G_{(\varepsilon_1,\varepsilon_2)}$, we see that $T_{\varepsilon_i}\leq R$.

Since $N$ is transitive on the vertices of $\Gamma$, we have
$G=NG_{(\varepsilon_1,\varepsilon_2)}$. Furthermore, since $N\leq G_i$, from the ``modular
law'', we obtain $G_i=NG_{(\varepsilon_1,\varepsilon_2)}\cap G_i=N(G_i\cap G_{(\varepsilon_1,\varepsilon_2)})$. Applying
$\pi_i$ on both sides of this equality, we get
$H=\pi_i(G_i)=\pi_i(T)\pi_i(G_i\cap G_{(\varepsilon_1,\varepsilon_2)})=TR$. Using again the ``modular
law'', we see that $H_{\varepsilon_i}=H_{\varepsilon_i}\cap TR=(H_{\varepsilon_i}\cap T)R=T_{\varepsilon_i}
R=R$.
\end{proof}

In what follows we use Lemma~\ref{Psfavourite} with $\varepsilon_1=\varepsilon_2=\delta$ (except in the proof of Lemma~\ref{abconstant-directed}, where we need it in its full generality).

The following two facts hardly deserve to be called lemmas, but will be used several times.

\begin{lemma}\label{vx-nbrs-adjacent}
If $\gamma_1, \gamma_2\in\Gamma^{-}(\gamma)$ and $d_H(\gamma_1,\gamma_2)=2$, then $\gamma_1$ and $\gamma_2$ are adjacent.
\end{lemma}

\begin{proof}
Either $(\gamma_1, \gamma_2) \in A^2_+\Gamma$, or $\gamma_1$ is adjacent to $\gamma_2$. But Lemma~\ref{n=2} says that if $(\gamma_1,\gamma_2) \in A^2_+\Gamma$, then $d_H(\gamma_1,\gamma_2)=1$, a contradiction.  
\end{proof}

\begin{lemma}\label{same-row-nonadjacent}
If $d_H(\gamma, \gamma')=1$, then $\gamma_1$ and $\gamma_2$ are not adjacent.
\end{lemma}

\begin{proof}This is simply a reminder of Remark~\ref{GpreservesHammingDist}.
\end{proof}

We require some information about neighbourhoods.  These will be used in our proofs of both the undirected and directed cases. Since each of these results applies (with the same proof) to any one of $\Gamma(\alpha')$, $\Gamma^+(\alpha')$ and $\Gamma^-(\alpha')$ for the appropriate choice of $\alpha'$, we introduce the notation $\Gamma^*(\alpha')$.  This notation
will be used to indicate that the result holds when ``$\Gamma^*$'' is replaced by any one of ``$\Gamma$'', ``$\Gamma^+$'', or ``$\Gamma^-$''.

We have $\beta=(\delta_1, \delta_2) \in \Gamma^-(\alpha)$ (in the undirected case, this is $\Gamma(\alpha)$).  Let $\gamma=(\delta_1', \delta_2') \in \Gamma^+(\alpha)$.  Clearly: 
\begin{eqnarray*}\Gamma(\alpha)=\beta^{G_\alpha}\subseteq
\beta^{W_\alpha}&=&\left(\delta_1^{H_\delta}\times
\delta_2^{H_\delta}\right)\cup\left(\delta_2^{H_\delta}\times
\delta_1^{H_\delta}\right);\\
\Gamma^-(\alpha)=\beta^{G_\alpha}\subseteq
\beta^{W_\alpha}&=&\left(\delta_1^{H_\delta}\times
\delta_2^{H_\delta}\right)\cup\left(\delta_2^{H_\delta}\times
\delta_1^{H_\delta}\right); \text{ and}\\
\Gamma^+(\alpha)=\gamma^{G_\alpha}\subseteq
\gamma^{W_\alpha}&=&\left((\delta_1')^{H_\delta}\times
(\delta_2')^{H_\delta}\right)\cup\left((\delta_2')^{H_\delta}\times
(\delta_1')^{H_\delta}\right).
\end{eqnarray*}
Thus $\Gamma(\alpha)$ (in the undirected case) or $\Gamma^-(\alpha)$ (in the directed case) is the
disjoint union of $N_\alpha$-orbits each of which is a ``rectangle'' with
$|\delta_1^{H_{\delta}}|$ rows and $|\delta_2^{H_{\delta}}|$ columns or with $|\delta_2^{H_{\delta}}|$ rows and $|\delta_1^{H_{\delta}}|$ columns. If $\delta_1$ and $\delta_2$ are in the same $H_\delta$-orbit then these two rectangles will instead be a single square.  Similarly, $\Gamma^+(\alpha)$ has the same structure.

\begin{lemma}\label{abconstant}
Fix any $\delta' \in \Delta$.  For each $\varepsilon \in \delta'^{H_{\delta}}$, the number $a$ of $\nu \in \Delta$ with $(\varepsilon, \nu) \in \Gamma^*(\alpha)$ is equal to the number of $\nu \in \Delta$ with $(\nu, \varepsilon) \in \Gamma^*(\alpha)$ and depends only on the $H_\delta$-orbit $\delta'^{H_\delta}$ (and not on the element $\varepsilon$).

Furthermore, if $(\varepsilon_1, \varepsilon_2) \in \Gamma^*(\alpha)$ and $\varepsilon_1^{H_\delta}=\varepsilon_2^{H_\delta}$, then $|\Gamma^*(\alpha)|=a|\varepsilon_1^{H_\delta}|$.
\end{lemma}

\begin{proof}
Fix $\varepsilon$ in $(\delta')^{H_\delta}$ and let $h_1$ be in $H_\delta$ with $\varepsilon^{h_1}=\delta'$. From Lemma~\ref{Psfavourite}, there exists $h_2\in H_\delta$ such that $g=(h_1,h_2)\in G_\alpha$. In particular, applying the automorphism $g$ we see that if $\nu_1,\ldots,\nu_a$ are the elements of $\Delta$ with $(\varepsilon,\nu_i)\in \Gamma^*(\alpha)$, then $\nu_1^{h_2},\ldots,\nu_a^{h_2}$ are exactly the elements of $\Delta$ with $(\delta',\nu_i^{h_2})\in \Gamma^*(\alpha)$. This shows that the number $a$ does not depend on the choice of $\varepsilon$ in $(\delta')^{H_\delta}$.

Now we show that there are exactly $a$ elements $\nu$ in $\Delta$ with $(\nu,\delta')\in \Gamma^*(\alpha)$. Let $\nu_1,\ldots,\nu_a$ be the elements of $\Delta$ with $(\delta',\nu_i)\in\Gamma^*(\alpha)$. Since $G_\alpha$ is transitive on $\{T_1,T_2\}$, there exists $x=(1\ 2)(t_1,t_2)\in G_\alpha$. As $t_2\in H_\delta$, from Lemma~\ref{Psfavourite} there exists $h_1\in H_\delta$ such that $y=(h_1,t_2^{-1})\in G_\alpha$. Now, $z=xy=(1\ 2)(t_1h_1,1)\in G_\alpha$ and $(\delta',\nu_i)^z=(\nu_i^{t_1h_1},\delta')$, for $i\in \{1,\ldots,a\}$, are exactly the elements in $\Gamma^*(\alpha)$ with second coordinate $\delta'$.

If $(\varepsilon_1, \varepsilon_2) \in \Gamma^*(\alpha)$ and $\varepsilon_1^{H_\delta}=\varepsilon_2^{H_\delta}$, then $\Gamma^*(\alpha) \subseteq \varepsilon_1^{H_\delta}\times \varepsilon_1^{H_\delta}$, and by Lemma~\ref{Psfavourite}, there are elements of $\Gamma^*(\alpha)$ with every possible first coordinate from $\varepsilon_1^{H_{\delta}}$.  By the earlier part of this lemma, there are exactly $a$ such elements for every possible first coordinate, making $|\Gamma^*(\alpha)|=a|\varepsilon_1^{H_\delta}|$, as claimed.
\end{proof}

\subsection{$\Gamma$ is undirected.}\label{und} We limit our attention to the undirected case first, which will prove easiest to complete.
We reserve the letters $a$ and $b$ to denote the numbers defined in Lemma~\ref{abconstant} that come from choosing $\delta'=\delta_1$ and $\delta'=\delta_2$, respectively. As $T_\delta$ does not fix either $\delta_1$ or $\delta_2$, we have $a,b\geq 2$.

A subset $X\subseteq V\Gamma$ is said to be {\em independent} if any two elements of $X$ are non-adjacent.

\begin{lemma}\label{ab-max-independent}
If $\gamma$ is any vertex of $\Gamma$ and $(\varepsilon_1,\varepsilon_2)=\gamma'\in \Gamma(\gamma)$, then $\{(\nu_1,\nu_2)\in\Gamma(\gamma): \nu_1=\varepsilon_1\}$ and $\{(\nu_1,\nu_2)\in\Gamma(\gamma): \nu_2=\varepsilon_2\}$ are the only maximal independent sets in $\Gamma(\gamma)$ containing $\gamma'$. Moreover, the cardinalities of these sets are $a$ and $b$ (respectively).
\end{lemma}

\begin{proof}
Write $X_1=\{(\nu_1,\nu_2)\in\Gamma(\gamma): \nu_1=\varepsilon_1\}$
and $X_2=\{(\nu_1,\nu_2)\in\Gamma(\gamma): \nu_2=\varepsilon_2\}$. Lemma~\ref{same-row-nonadjacent} shows that $X_1$ and $X_2$ are independent sets (both containing $\gamma'$), and Lemma~\ref{vx-nbrs-adjacent} implies that no other vertex of $\Gamma(\gamma)$ is independent from $\gamma'$, so these independent sets are maximal and there are no others. The cardinality follows from Lemma~\ref{abconstant}.
\end{proof}

\begin{corollary}\label{indep-sets}
If $b\neq a$, then for each vertex $\gamma$, the neighbourhood $\Gamma(\gamma)$  can be uniquely decomposed into a disjoint union of independent sets of cardinality $b$.
\end{corollary}

\begin{proof}
Since $b \neq a$, Lemma~\ref{ab-max-independent} says that every neighbour of $\gamma$ lies in a unique maximal independent set of cardinality $b$.  The uniqueness means that these sets must be disjoint.  The result follows.
\end{proof}

\begin{lemma}\label{a-ge-2-nbrs}
The vertices of $\Gamma(\beta) \setminus \Gamma(\alpha)$ are: $\alpha$, $(\delta, \nu)$ for every $\nu$ such that $(\delta_1,\nu) \in \Gamma(\alpha)$, and  $(\nu, \delta)$ for every $\nu$ such that $(\nu,\delta_2) \in \Gamma(\alpha)$.
\end{lemma}

\begin{proof}
From Lemmas~\ref{vx-nbrs-adjacent} and~\ref{same-row-nonadjacent}, we see that the elements of $\Gamma(\beta)\setminus(\Gamma(\alpha)\cup\{\alpha\})$ are of the form $(\delta,\nu)$ or $(\nu,\delta)$, for some $\nu\in \Delta\setminus\{\delta\}$. Let $\nu$ be in $\Delta$ with $(\delta,\nu)\in \Gamma(\beta)$. We need to show that $(\delta_1,\nu)\in \Gamma(\alpha)$. We argue by contradiction, so we assume that $(\delta_1,\nu)\not\in\Gamma(\alpha)$. Write $X_\alpha=\{\eta\in\Delta : (\eta,\nu)\in \Gamma(\alpha)\}$ and $X_\beta=\{\eta\in\Delta : (\eta,\nu)\in \Gamma(\beta)\}$.

Lemma~\ref{ab-max-independent} shows that $|X_{\beta}|$ is either $a$ or $b$.  Since $a$ and $b$ are each at least 2, $|X_{\beta}| \ge 2$, so Lemma~\ref{vx-nbrs-adjacent} implies that $X_\alpha \neq \emptyset$.  Hence 
$|X_\alpha|$ is also either $a$ or $b$. Now, replacing $a$ by $b$ if necessary, we may assume that $b\geq a$. Since $\alpha$ and $\beta$ are adjacent, Lemmas~\ref{vx-nbrs-adjacent} and~\ref{same-row-nonadjacent} yield that $X_\alpha\setminus\{\delta_1\}\subseteq X_\beta$  and  $X_\beta\setminus\{\delta\}\subseteq X_\alpha$. (Because if $(\mu, \nu) \in \Gamma(\alpha)$ and $\mu \neq \delta_1$, then since $(\delta, \nu) \in \Gamma(\beta)$ we have $\nu \neq \delta_2$, so $\beta \in \Gamma(\alpha)$ implies $(\mu,\nu) \in \Gamma(\beta)$.) As we are assuming that $\delta_1\not\in X_\alpha$, we get $X_\beta= X_\alpha\cup\{\delta\}$ so $|X_\beta|=|X_\alpha|+1$. Since $b\geq a$, we obtain $|X_\beta|=b$, $|X_\alpha|=a$ and $b=a+1$.

Write $X_\alpha=\{x_1,\ldots,x_a\}$, $\gamma=(\delta,\nu)$ and $\gamma_i=(x_i,\nu)$, for $i\in \{1,\ldots,a\}$. From the previous paragraph, $\{\gamma_1,\ldots,\gamma_a\}$ is a maximal independent set of $\Gamma(\alpha)$ of size $a$ and $\{\gamma,\gamma_1,\ldots,\gamma_a\}$ is a maximal independent set of $\Gamma(\beta)$ of size $b$. So, from Lemma~\ref{ab-max-independent} applied to $\alpha$ and $\gamma_i$ (for each $i\in \{1,\ldots,a\}$), we see that there exists $Y_i\subseteq \Delta$ of size $b$ such that 
$V_i=\{(x_i,y) : y\in Y_i\}$ is a maximal independent set of $\Gamma(\alpha)$. If $\delta_2\not\in Y_i$, then by Lemmas~\ref{vx-nbrs-adjacent} and~\ref{same-row-nonadjacent} the set $V_i$ is contained in $\Gamma(\beta)$. Therefore $V_i$ and $\{\gamma,\gamma_1,\ldots,\gamma_a\}$ are both independent sets of $\Gamma(\beta)$ of size $b$ containing $\gamma_i$, which contradicts Lemma~\ref{ab-max-independent}. Thus $\delta_2\in Y_i$, for each $i\in \{1,\ldots,a\}$, so $(x_i, \delta_2) \in \Gamma(\alpha)$ for every $i \in \{1, \ldots, a\}$.

Now, $V_1$ and $\{\beta,(x_1,\delta_2),\ldots,(x_a,\delta_2)\}$ are independent sets of $\Gamma(\alpha)$ of size $b$ both containing $(x_1,\delta_2)$, again contradicting Lemma~\ref{ab-max-independent}. This final contradiction gives that $(\delta_1,\nu)\in\Gamma(\alpha)$.

The proof for the neighbours of $\beta$ of the form $(\nu,\delta)$ is entirely symmetric. 
\end{proof}

\begin{lemma}\label{a-at-least-2}
$\Gamma$ has diameter $2$.
\end{lemma}

\begin{proof}
Towards a contradiction, suppose that $\gamma$ is a vertex at distance $3$ from $\alpha$.  Then $\gamma$ has a neighbour $\alpha'$ that is at distance 2 from $\alpha$.  By Lemma~\ref{n=2}, $d_H(\alpha',\alpha)=1$. 

Let $\sigma$ be an arbitrary mutual neighbour of $\alpha$ and $\alpha'$.  Then $\sigma$ can take the role of $\beta$ in Lemma~\ref{a-ge-2-nbrs}, which means that since $\alpha' \in \Gamma(\sigma)\setminus \Gamma(\alpha)$, it must be the case that the unique vertex that lies at Hamming distance 1 from both $\sigma$ and $\alpha'$ but at Hamming distance 2 from $\alpha$, is in $\Gamma(\alpha)$.
Thus, the mutual neighbours of $\alpha$ and $\alpha'$ can be found in the following manner:  first, choose any vertex $\tau \in \Gamma(\alpha)$ with $d_H(\tau, \alpha')=1$; then (by Lemma~\ref{a-ge-2-nbrs}), any vertex $\sigma \in \Gamma(\alpha)$ with $d_H(\sigma,\tau)=1$ but $d_H(\sigma,\alpha')=2$, will be a mutual neighbour of $\alpha$ and $\alpha'$.  

Since $\gamma$ is at distance $3$ from $\alpha$, and any such $\sigma$ is adjacent to $\alpha$ and to $\alpha'$, it must be the case that $\sigma$ is at distance $2$ from $\gamma$. Then by Lemma~\ref{n=2}, for any such $\sigma$, $d_H(\sigma,\gamma)=1$. 

We have either $a$ or $b$ choices for $\tau$ (vertices that are in $\Gamma(\alpha)$ at Hamming distance $1$ from $\alpha'$).  For each choice of $\tau$, we have either $b-1$ or $a-1$ choices for $\sigma$ (vertices that are in $\Gamma(\alpha)$, at Hamming distance $1$ from $\tau$ and at Hamming distance $2$ from $\alpha'$).  Thus, there are either $b(a-1)$ or $a(b-1)$ possible choices for $\sigma$.  However every such $\sigma$ is in $\Gamma(\alpha')$ together with $\gamma$. By Lemmas~\ref{same-row-nonadjacent} and~\ref{ab-max-independent}, there can be at most $a-1+b-1=a+b-2$ choices for $\sigma$.  The inequality $b(a-1)\le a+b-2$ can be solved only if $a=2$, while the inequality $a(b-1) \le a+b-2$ can be solved only if $b=2$.  So, replacing $a$ by $b$ if necessary, we may assume that $a=2$. In particular, $b\geq a$. Furthermore, without loss of generality, we can assume that there are $b$ choices for $\tau$, and for each of these there is a unique choice for $\sigma$. (We denote by $\sigma_\tau$ the choice of $\sigma$ determined by $\tau$.)  

Using $\wedge$-transitivity, we may assume that $\alpha'=(\delta,\delta')$ (recall that $\alpha=(\delta,\delta)$). Notice that the $b$ choices for $\tau$  all have the same value in their $2^{\mathrm{nd}}$ entry because they are at Hamming distance $2$ from $\alpha$ and at Hamming distance $1$ from $\alpha'$.  So they have $b$ distinct values in their $1^{\textrm{st}}$ entry. However, since $d_H(\sigma_\tau,\alpha')=2$ and $d_H(\sigma_\tau,\tau)=1$, we obtain that $\tau$ and $\sigma_\tau$ have the same $1^{\textrm{st}}$ entry. Moreover, for each $\sigma_\tau$, $d_H(\sigma_\tau,\gamma)=1$ only 
 if $\sigma_\tau$ has the same $2^{\textrm{nd}}$ entry as $\gamma$, for each $\tau$.  But in this case, these $b$ choices for $\sigma_\tau$, together with $\gamma$, form an independent set of cardinality $b+1$ in $\Gamma(\alpha')$, contradicting Lemma~\ref{ab-max-independent}.
\end{proof}

The proof of Theorem~\ref{main-thm} when $\Gamma$ is undirected is now easy.
\begin{corollary}\label{co1}
$\Gamma$ is isomorphic to the complement of $\H(m,2)$.
\end{corollary}
\begin{proof}
From Lemma~\ref{a-at-least-2}, we conclude that there are no vertices at distance $3$ from $\alpha$.  
Since $\Gamma$ is connected, every vertex is at distance 1 or 2 from $\alpha$. Since $G$ is $\wedge$-transitive and preserves Hamming distance, every vertex at Hamming distance $2$ from $\alpha$ is adjacent to $\alpha$ in $\Gamma$, and every vertex at Hamming distance $1$ from $\alpha$ is at distance $2$ from $\alpha$ in $\Gamma$.  But then $\Gamma$ is isomorphic to the complement of the Hamming graph $\H(m,2)$.
\end{proof}

\subsection{$\Gamma$ is directed}\label{dir} In this subsection we assume that $\Gamma$ is directed and we conclude the proof of Theorem~\ref{main-thm}. We let $k$ denote the number of out-neighbours of any vertex (so the total valency of a vertex is $2k$). Note that for any vertex $v$, we have $\Gamma^+(v) \cap \Gamma^-(v)=\emptyset$, since otherwise by arc-transitivity $\Gamma$ is undirected. Furthermore, we fix $\gamma=(\delta_1',\delta_2')\in \Gamma^+(\alpha)$.

%The first ideas we will need lead to a strengthening of Lemma~\ref{abconstant}. 

\begin{lemma}\label{same-orbit-lengths}
For $\eta$ and $\nu$ in $\Delta$, we have $|\eta^{H_{\nu}}|=|\nu^{H_{\eta}}|$.
\end{lemma} 

\begin{proof}
A double counting gives $|\Delta||\nu^{H_\eta}|=|(\eta,\nu)^H|=|\Delta||\eta^{H_\nu}|$ (see~\cite[Theorem~$16.3$]{Wie}). Since $\Delta$ is finite and nonempty, this concludes our proof.
\end{proof}

We will need the following fact in a few places.

\begin{lemma}\label{only-2-cases}
We have $\delta_1^{H_{\delta}}=\delta_2^{H_{\delta}}$ if and only if $\delta_1'^{H_{\delta}}=\delta_2'^{H_{\delta}}.$
\end{lemma}

\begin{proof}
If $\delta_1^{H_{\delta}}=\delta_2^{H_{\delta}}$, then there is some $h \in H_{\delta}$ with $\delta_1^h=\delta_2$.  Arc-transitivity means that there is some $g\in G$ with $(\beta,\alpha)^g=(\alpha,\gamma)$. Now, $g=\sigma(h_1, h_2)$ with $h_1,h_2\in H$ and $\sigma=1$ or $\sigma=(1\ 2)$. As $\alpha^g=\gamma$, we must have $\delta^{h_1}=\delta_1'$ and $\delta^{h_2}=\delta_2'$. 

If $\sigma=1$, from $\beta^g=\alpha$ we  have $\delta_1^{h_1}=\delta$ and $\delta_2^{h_2}=\delta$.  It is now clear that $h_2^{-1}h^{-1}h_1 \in H_{\delta}$, and $(\delta_2')^{h_2^{-1}h^{-1}h_1}=\delta_1'$, completing the proof in this case.
If $\sigma=(1\ 2)$, from $\beta^g=\alpha$ we  have $\delta_2^{h_1}=\delta$ and $\delta_1^{h_2}=\delta$. 
It is clear that $h_1^{-1}h^{-1}h_2 \in H_{\delta}$, and $(\delta_1')^{h_1^{-1}h^{-1}h_2}=\delta_2'$, completing the proof.  

The converse is analogous.
\end{proof}

The next result nicely limits the cases that we need to consider.

\begin{lemma}\label{square}
If $(\varepsilon_1, \delta),  (\delta, \varepsilon_2) \in \Gamma^-(\gamma)$ with $\varepsilon_1, \varepsilon_2 \neq \delta$, then $\varepsilon_1^{H_{\delta}}=\varepsilon_2^{H_{\delta}}$. Moreover, $\delta_1'^{H_{\delta}}=\delta_2'^{H_{\delta}}$, and either $\varepsilon_1^{H_\delta}=\delta_1'^{H_\delta}$, or $\varepsilon_1^{H_\delta}=\delta_1^{H_\delta}$. 
\end{lemma}

\begin{proof}
We start by proving that if $\alpha_1=(\varepsilon_1,\delta)$, $\alpha_2=(\delta,\varepsilon_2)$ and $\alpha_1,\alpha_2\in \Gamma^-(\gamma)$, then $\varepsilon_1^{H_\delta}=\varepsilon_2^{H_\delta}$. As $d_H(\alpha_1,\alpha)=d_H(\alpha_2,\alpha)=1$, we see from Lemma~\ref{same-row-nonadjacent} that $\alpha_1$ and $\alpha_2$ are not adjacent to $\alpha$, that is, $(\alpha,\alpha_1),(\alpha,\alpha_2)\in A^2_+\Gamma$. So, by $\wedge$-transitivity, there must be some element $g \in G_{\alpha}$ with $\alpha_1^g=\alpha_2$.  It is not hard to see that $g=(1\ 2)(h_1, h_2)$, $h_1,h_2\in H_\delta$ and $\varepsilon_1^{h_2}=\varepsilon_2$.  Hence $\varepsilon_1$ and $\varepsilon_2$ are in the same $H_{\delta}$-orbit, that is, $\varepsilon_1^{H_\delta}=\varepsilon_2^{H_\delta}$.

Certainly, \[(\dag)\qquad\Gamma^-(\gamma) \supseteq \alpha^{N_\gamma}=\delta^{T_{\delta_1'}}\times \delta^{T_{\delta_2'}}.\]  Let $\varepsilon_1 \in \delta^{T_{\delta_1'}}\setminus\{\delta\}$ and $\varepsilon_2 \in \delta^{T_{\delta_2'}}\setminus\{\delta\}$. (Note that $\varepsilon_1$ and $\varepsilon_2$ are well-defined because $T_{\delta_i'}$ does not fix $\delta$.) As $(\varepsilon_1,\delta),(\delta,\varepsilon_2)\in \Gamma^-(\gamma)$, from the previous paragraph, we have  $\varepsilon_1^{H_\delta}=\varepsilon_2^{H_\delta}$. We also have from~$(\dag)$ that $\alpha'=(\varepsilon_1, \varepsilon_2) \in \Gamma^-(\gamma)$ has $d_H(\alpha',\alpha)=2$. So by Lemma~\ref{vx-nbrs-adjacent}, $\alpha$ and $\alpha'$ are adjacent.  
If $\alpha' \in \Gamma^+(\alpha)$, then by arc-transitivity, $\Gamma^+(\alpha) =\alpha'^{G_\alpha}\subseteq \alpha'^{W_\alpha}=\varepsilon_1^{H_{\delta}}\times\varepsilon_1^{H_{\delta}}$ (since $\varepsilon_1$ and $\varepsilon_2$ are in the same $H_{\delta}$-orbit). As $\gamma\in\Gamma^{+}(\alpha)$ and $\gamma=(\delta_1',\delta_2')$, we obtain $\delta_1', \delta_2' \in \varepsilon_1^{H_{\delta}}$ and $\delta_1'^{H_\delta}=\delta_2'^{H_\delta}$.
On the other hand, if $\alpha' \in \Gamma^-(\alpha)$, then an analogous argument yields $\Gamma^-(\alpha) \subseteq \varepsilon_1^{H_{\delta}}\times \varepsilon_1^{H_\delta}$, so $\delta_1^{H_\delta}=\delta_2^{H_\delta}=\varepsilon_1^{H_\delta}$, and Lemma~\ref{only-2-cases} completes the proof.
\end{proof}

Now we can obtain an extension of Lemma~\ref{abconstant}.

\begin{lemma}\label{abconstant-directed}
Let $\gamma'=(\delta_1',\delta_1')$.
For each $\varepsilon_1\in \delta_1'^{H_\delta}$ and $\varepsilon_2 \in \delta^{H_{\delta_1'}}$, the cardinalities of the following sets are equal and do not depend on $\varepsilon_1$ or $\varepsilon_2$:
\begin{enumerate}\renewcommand{\theenumi}{\alph{enumi}}
\item $\{ \nu \in \Delta : (\varepsilon_1, \nu) \in \Gamma^+(\alpha)\}$;
\item $\{ \nu \in \Delta : (\nu, \varepsilon_1) \in \Gamma^+(\alpha)\}$;
\item $\{ \nu \in \Delta : (\varepsilon_2, \nu) \in \Gamma^-(\gamma')\}$;
\item $\{ \nu \in \Delta : (\nu, \varepsilon_2) \in \Gamma^-(\gamma')\}$.
\end{enumerate}
Furthermore, $k=a|\delta_1'^{H_{\delta}}|$, where $a$ is the cardinality of each of these sets, and $a \ge 2$.
\end{lemma}

\begin{proof}
The equality of the cardinalities of the sets in~(a) and~(b) follows immediately from Lemma~\ref{abconstant}.  Replacing $\alpha$ by $\gamma'$ in Lemma~\ref{abconstant} shows the equality of the cardinalities of the sets in~(c) and~(d).  Call the first of these cardinalities $a$, and the second $a'$.  By Lemma~\ref{square} we have $\delta_1'^{H_\delta}=\delta_2'^{H_{\delta}}$, so since $\gamma=(\delta_1',\delta_2') \in \Gamma^+(\alpha)$, Lemma~\ref{abconstant} tells us that $k=a|\delta_1'^{H_\delta}|$.  

We will now find an in-neighbour of $\gamma'$ whose second entry is $\delta$, and whose first entry is in $\delta^{H_{\delta_1'}}$.  Then Lemma~\ref{abconstant} applied to this in-neighbour shows that $k=a'|\delta^{H_{\delta_1'}}|$.  Since $|\delta_1'^{H_\delta}|=|\delta^{H_{\delta_1'}}|$ by Lemma~\ref{same-orbit-lengths}, this will show that $a=a'$.

Let $t \in H_{\delta}$ such that $(\delta_2')^t=\delta_1'$ (we can do this since $(\delta_2')^{H_\delta}=(\delta_1')^{H_\delta}$ by Lemma~\ref{square}). Let $\alpha''=(\delta_1',\delta)$. By Lemma~\ref{Psfavourite}, there must be some $g=(t',t) \in G_{\alpha''}$.  Applying $g$ to $\alpha$ and $\gamma$, we get $\gamma^g=(\delta_1',\delta_1')$ and $\alpha^g=(\delta^{t'},\delta)$.  This shows that $\alpha^g$ is an in-neighbour of $\gamma'$ with the desired form.

That $a \ge 2$ follows from the fact that $T_{\delta_1'}$ does not fix $\delta$.
\end{proof}

For the rest of this subsection, we let $a$ denote the constant defined in Lemma~\ref{abconstant-directed}.

\begin{lemma}\label{k-a-b+1}
The sets $\Gamma^+(\beta')\cap\Gamma^-(\alpha)$ and  $\Gamma^-(\beta')\cap\Gamma^-(\alpha)$ each have cardinality $(k-2a+1)/2$ for any $\beta' \in \Gamma^-(\alpha)$.
\end{lemma}

\begin{proof}
If $\beta' \in \Gamma^-(\alpha)$ with $d_H(\beta',\beta)=1$, then by Lemma~\ref{same-row-nonadjacent}, there cannot be an arc between $\beta$ and $\beta'$.  However, if $\beta' \in \Gamma^-(\alpha)$ with $d_H(\beta',\beta)=2$, then by Lemma~\ref{vx-nbrs-adjacent}, there must be an arc between $\beta$ and $\beta'$. Using the fact that there cannot be arcs in both directions between $\beta$ and any other vertex, we conclude that $\beta$ has precisely $k-2a+1$ arcs to or from other vertices in $\Gamma^-(\alpha)$.  Since $G$ is arc-transitive, every vertex of $\Gamma^-(\alpha)$ has the same number of out-neighbours in $\Gamma^-(\alpha)$ as every other vertex; also, every vertex of $\Gamma^-(\alpha)$ has the same number of in-neighbours in $\Gamma^-(\alpha)$ as every other vertex.  This shows that the total number of arcs both of whose endpoints lie within $\Gamma^-(\alpha)$ is $k(k-2a+1)/2$ (we divide by two since each arc has been counted at both ends).  Our conclusions are immediate.
\end{proof}

We can now generalize Lemma~\ref{vx-nbrs-adjacent} to vertices that share an in-neighbour.

\begin{lemma}\label{vx-in-nbrs-adjacent}
If $\gamma_1$ and $\gamma_2$ share an in-neighbour and $d_H(\gamma_1,\gamma_2)=2$, then they must be adjacent.
\end{lemma}

\begin{proof}
Call the shared in-neighbour $\alpha'$.  We will show that $\gamma_1$ has $(k-2a+1)/2$ out-neighbours and $(k-2a+1)/2$ in-neighbours in $\Gamma^+(\alpha')$.  Since there are $2a-2$ vertices in $\Gamma^+(\alpha')$ that are at Hamming distance 1 from $\gamma_1$ (by Lemma~\ref{abconstant-directed}), and $d_H(\gamma_1,\gamma_1)\neq 2$, there must be $k-2a+1$ vertices in $\Gamma^+(\alpha')$ that are at Hamming distance 2 from $\gamma_1$, so this count will show that all of these vertices are adjacent to $\gamma_1$, which yields the conclusion.

By Lemma~\ref{k-a-b+1} and arc-transitivity, $|\Gamma^+(\alpha')\cap \Gamma^-(\gamma_1)|=(k-2a+1)/2$.  Consider the induced subgraph on $\Gamma^+(\alpha')$.  Since $G_{\alpha'}$ is transitive on this set, the in-valency and out-valency of every vertex is constant in this subgraph, so every vertex has in-valency and out-valency $(k-2a+1)/2$, since $\gamma_1$ has this in-valency.
\end{proof}

\begin{lemma}\label{more-k-a-b+1}
For any $\gamma' \in \Gamma^+(\alpha)$, % and any $\beta' \in \Gamma^-(\alpha)$, 
the sets  $\Gamma^-(\gamma')\cap\Gamma^+(\alpha)$ and  $\Gamma^+(\gamma')\cap\Gamma^+(\alpha)$
have cardinality $(k-2a+1)/2$.
%$\Gamma^-(\gamma')\cap\Gamma^+(\alpha)$, $\Gamma^+(\gamma')\cap\Gamma^+(\alpha)$
%\item $\Gamma^+(\gamma)\cap\Gamma^-(\alpha)$;
%\item $\Gamma^-(\gamma)\cap\Gamma^-(\alpha)$;
%\item $\Gamma^+(\beta')\cap\Gamma^+(\alpha)$;
%\item $\Gamma^-(\gamma')\cap\Gamma^-(\alpha)$.
%\end{enumerate}
\end{lemma}

\begin{proof}
Replacing $\Gamma^-(\alpha)$ by $\Gamma^+(\alpha)$ and $\beta$ by $\gamma'$ throughout the proof of Lemma~\ref{k-a-b+1}, with Lemma~\ref{vx-nbrs-adjacent} replaced by Lemma~\ref{vx-in-nbrs-adjacent} yields the desired conclusion.
%
%Consider the arc $(\alpha, \gamma')$.  We have just shown (by conclusion (b)) that this arc lies in precisely $(k-2a+1)/2$ triangles whose other two arcs have the form $(\gamma',\gamma'')$, $(\alpha,\gamma'')$ (since $\gamma''$ has to lie in $\Gamma^+(\alpha)\cap \Gamma^+(\gamma')$ in such a triangle).  Since $\Gamma$ is arc-transitive, $(\beta', \alpha)$ lies in precisely $(k-2a+1)/2$ triangles whose other two arcs have the form $(\alpha, \beta''), (\beta', \beta'')$.  Since this happens if and only if $\beta'' \in \Gamma^+(\alpha)\cap \Gamma^+(\beta')$, conclusion (c) is immediate.
%
%Similarly, the arc $(\beta', \alpha)$ (by conclusion (b) of Lemma~\ref{k-a-b+1}) lies in precisely $(k-2a+1)/2$ triangles whose other two arcs have the form $(\beta'', \beta')$, $(\beta'', \alpha)$ (since $\beta''$ has to lie in $\Gamma^-(\alpha)\cap \Gamma^-(\beta')$ in such a triangle).  Since $\Gamma$ is arc-transitive, $(\alpha, \gamma')$ lies in precisely $(k-2a+1)/2$ triangles whose other two arcs have the form $(\gamma'', \alpha)$, $(\gamma'', \gamma')$.  Since this happens if and only if $\gamma'' \in \Gamma^-(\alpha) \cap \Gamma^-(\gamma')$, conclusion (d) is immediate.
\end{proof}

\begin{corollary}\label{gamma-nbrs}
If $\beta' \in \Gamma^-(\alpha)$, then $\Gamma^+(\beta')\subset \Gamma^+(\alpha) \cup \Gamma^-(\alpha) \cup HD_1(\alpha)$, where $HD_1(\alpha)$ is the set of vertices at Hamming distance 1 from $\alpha$.  

Also, if $\gamma' \in \Gamma^+(\alpha)$, then $\Gamma^-(\gamma')\subset \Gamma^+(\alpha) \cup \Gamma^-(\alpha) \cup HD_1(\alpha)$.
\end{corollary}

\begin{proof}
We have $|\Gamma^+(\beta')|=k$.  The sets $\Gamma^+(\alpha)$ and $\Gamma^-(\alpha)$ are disjoint since we are in the directed case, and Lemmas~\ref{k-a-b+1} and \ref{more-k-a-b+1} together with arc-transitivity tell us that  $|\Gamma^+(\beta')\cap\Gamma^+(\alpha)|= |\Gamma^+(\beta')\cap\Gamma^-(\alpha)|=(k-2a+1)/2$, so this accounts for all but $2a-1$ of the out-neighbours of $\beta'$.  But $\alpha \in \Gamma^+(\beta')$, and by Lemma~\ref{abconstant-directed} with vertex-transitivity, we see that for $i= 1, 2$ there must be precisely $a-1$ other out-neighbours of $\beta'$ that have the same entry as $\alpha$ in coordinate $i$.  By Lemma~\ref{same-row-nonadjacent}, none of these vertices is in either $\Gamma^+(\alpha)$ or $\Gamma^-(\alpha)$, so these together with $\alpha$ itself form the remaining $2a-1$ out-neighbours of $\beta'$.

The proof for $\gamma'$ is analogous.
\end{proof}

\begin{lemma}\label{delta-orbits-equal}
Suppose that $\delta_1^{H_{\delta}}=\delta_1'^{H_{\delta}}$.  Then for any $\delta' \in \delta_1^{H_{\delta}}$, we have $\delta^{H_{\delta'}}\setminus \{\delta\}=\delta'^{H_{\delta}}\setminus\{\delta'\}$.
\end{lemma}

\begin{proof}
We will show that $\delta^{H_{\delta'}}=(\delta'^{H_{\delta}}\cup \{\delta\}) \setminus\{\delta'\}$.  Clearly $\delta \in \delta^{H_{\delta'}}\setminus\delta'^{H_{\delta}}$ and $\delta' \in \delta'^{H_{\delta}}\setminus\delta^{H_{\delta'}}$, and the cardinalities of the two orbits are equal (by Lemma~\ref{same-orbit-lengths}), so if we can show that $\delta^{H_{\delta'}}\subset \delta'^{H_{\delta}}\cup\{\delta\}$, that will be sufficient.  

Since $\gamma=(\delta_1',\delta_2') \in \Gamma^+(\alpha)$ and $\delta'\in \delta_1'^{H_{\delta}}$, Lemma~\ref{Psfavourite} tells us that there is some element of $G_{\alpha}$ that fixes the first coordinate (so fixes each of the two coordinates) and takes $\gamma$ to some vertex $\gamma'$ whose first entry is $\delta'$.  Clearly, $\gamma' \in \Gamma^+(\alpha)$.

A similar argument shows that there is some in-neighbour of $\alpha$ whose first entry is any fixed element of $\delta_1^{H_{\delta}}$.   Thus, there exists some in-neighbour of $\alpha$ and some out-neighbour of $\alpha$ whose first entries are any fixed element of $\delta_1^{H_{\delta}}$.  
In fact, Lemma~\ref{abconstant-directed} tells us that there exists $a$ in-neighbours and $a$ out-neighbours of $\alpha$ in each of these columns.  These are in fact all of the in- and out-neighbours of $\alpha$, since $k=a|\delta_1^{H_{\delta}}|$.

By Corollary~\ref{gamma-nbrs}, $\Gamma^-(\gamma') \subset \Gamma^+(\alpha)\cup \Gamma^-(\alpha)\cup HD_1(\alpha)$.  Now, we have just concluded that any neighbour of $\alpha$ must have its first entry in the set $\delta_1^{H_{\delta}}$.  The in-neighbours of $\gamma'$, therefore, must have their first entries in the set $(\delta_1^{H_{\delta}} \cup\{\delta\})\setminus\{\delta'\}$.  
But since $\alpha$ is an in-neighbour of $\gamma'$, Lemma~\ref{Psfavourite} tells us that there is an element of $G$ that fixes $\gamma'$, fixes the coordinates, and takes the column containing $\delta$ to the column indexed by any element of $\delta^{H_{\delta'}}$.  
So these indices must be elements of $(\delta_1^{H_{\delta}}\cup\{\delta\})\setminus\{\delta'\},$ 
meaning that we must have $\delta^{H_{\delta'}} \subset  \delta'^{H_{\delta}} \cup\{\delta\}$, as desired.
\end{proof}

\begin{lemma}\label{one-square-2-trans}
Suppose $\delta_1^{H_{\delta}}=\delta_2^{H_{\delta}}=\delta_1'^{H_{\delta}}=\delta_2'^{H_{\delta}}$.  Then $H$ is $2$-transitive on $\Delta$.
\end{lemma}

\begin{proof}
We claim that if $\Delta \neq  \delta_1'^{H_{\delta}}\cup \{\delta\}$, then $\Gamma$ is disconnected. This will be a contradiction, so we conclude that $\Delta = \delta_1'^{H_{\delta}}\cup \{\delta\}$, 
%and 
%by Lemma~\ref{delta-orbits-equal}, we have shown that for every $\delta' \in \Delta$, $H_{\delta'}$ is transitive on $\Delta\setminus\{\delta'\}$, 
which forces $H$ to be 2-transitive on $\Delta$, completing the proof.

If we can prove that whenever $\varepsilon$ is in $(\delta_1'^{H_{\delta}}\cup \{\delta\})\times (\delta_1'^{H_{\delta}}\cup \{\delta\})$ and has some in- or out-neighbour in this set, all of its in- and out-neighbours must be in this set, this will establish the claim we made in the preceding paragraph, and so complete the proof.  
Let $\varepsilon=(\delta_3, \delta_4)$ and $\mu=(\delta_3', \delta_4')$, where $\delta_3, \delta_4, \delta_3', \delta_4' \in \delta_1'^{H_{\delta}}\cup \{\delta\}$, and suppose that $\mu$ is either an in-neighbour or out-neighbour of $\varepsilon$.  Let us suppose that $\mu \in \Gamma^-(\varepsilon)$. 
By arc-transitivity, we have $\Gamma^-(\varepsilon) = \mu^{G_\varepsilon} \subseteq \left( (\delta_3')^{H_{\delta_3}} \times (\delta_4')^{H_{\delta_4}}\right) \cup \left( (\delta_4')^{H_{\delta_3}} \times (\delta_3')^{H_{\delta_4}}\right)$. Using Lemma~\ref{delta-orbits-equal} it is straightforward to verify that this is a subset of $\left((\delta_1')^{H_{\delta}} \cup \{\delta\}\right)\times \left((\delta_1')^{H_{\delta}} \cup \{\delta\}\right)$. It remains to show that $\Gamma^+(\varepsilon)$ is also in this set.
Notice that a similar argument shows that $\Gamma^+(\mu) \subseteq (\delta_1'^{H_{\delta}}\cup \{\delta\})\times (\delta_1'^{H_{\delta}}\cup \{\delta\})$ since $\mu$ has one out-neighbour (namely $\varepsilon$) in this set.  But by Lemma~\ref{more-k-a-b+1} and arc-transitivity, $\Gamma^+(\mu) \cap \Gamma^+(\varepsilon)$ has cardinality $(k-2a+1)/2$.  Since Lemma~\ref{abconstant-directed} shows that $a$ divides $k$ and $a \ge 2$, we must have $(k-2a+1)/2>0$, so $\varepsilon$ has at least one out-neighbour in $(\delta_1'^{H_{\delta}}\cup \{\delta\})\times (\delta_1'^{H_{\delta}}\cup \{\delta\})$, from which a similar argument shows that all out-neighbours of $\varepsilon$ are in this set.  

The case in which $\mu \in \Gamma^+(\varepsilon)$ is precisely analogous to the above case.
\end{proof}

\begin{lemma}\label{one-square-done}
It is not possible to have $\delta_1^{H_{\delta}}=\delta_2^{H_{\delta}}=\delta_1'^{H_{\delta}}=\delta_2'^{H_{\delta}}$.
\end{lemma}

\begin{proof}
Towards a contradiction, suppose that $\delta_1^{H_{\delta}}=\delta_2^{H_{\delta}}=\delta_1'^{H_{\delta}}=\delta_2'^{H_{\delta}}$.
By Lemma~\ref{k-a-b+1} we have $(k-2a+1)/2 \in \mathbb{Z}$, so $k$ must be odd.  But $k=a|\delta_1^{H_{\delta}}|$ (by Lemma~\ref{abconstant-directed}), so $a$ and $|\delta_1^{H_{\delta}}|$ are both odd.  Then by Lemma~\ref{one-square-2-trans}, $|\Delta|=|\delta_1^{H_{\delta}}|+1$, so $|\Delta|$ must be even.

Now, $T$ is a transitive group acting on $\Delta$, so $T$ must have even order.  Hence $T$ contains an involution $t$.  Without loss of generality, we can assume that $\delta^t=\delta' \neq \delta$, and $\delta'^t=\delta$.  Since $\Delta=\delta_1^{H_\delta} \cup \{\delta\}$, we have $\delta' \in \delta_1^{H_{\delta}}=(\delta_1')^{H_\delta}$. Now by Lemma~\ref{abconstant-directed}, $\alpha$ has $a$ out-neighbours whose first entry is $\delta'$; we choose one of these, $(\delta', \delta'')$.

Since $H$ is 2-transitive on $\Delta$ (by Lemma~\ref{one-square-2-trans}), there exists $h \in H_{\delta}$ such that $\delta'^h=\delta''$.  Now since $T \triangleleft H$, we have $h^{-1}th \in T$.  We know that $T \times T=N \leq G$; consider the action of $g=(t, h^{-1}th) \in G$ on $\alpha$ and on $(\delta', \delta'')$.  We have $(\delta, \delta)^{(t, h^{-1}th)}=(\delta', \delta'')$ since $h \in H_{\delta}$.  And $(\delta', \delta'')^{(t, h^{-1}th)}=(\delta, (\delta')^{th})=(\delta, \delta^h)=(\delta,\delta)$.  So $g$ reverses this arc, contradicting the fact that we are in the directed case.
\end{proof}

By Lemma~\ref{square} we have $(\delta_1')^{H_\delta}=(\delta_2')^{H_\delta}$, so by Lemma~\ref{only-2-cases} $\delta_1^{H_\delta}=\delta_2^{H_\delta}$ and the hypothesis eliminated in our next  lemma is the only remaining possibility.

\begin{lemma}\label{two-squares-done}
It is not possible to have $\delta_1^{H_{\delta}}\neq \delta_1'^{H_{\delta}}.$
\end{lemma}

\begin{proof}
By Corollary~\ref{gamma-nbrs}, $\Gamma^-(\gamma) \subset \Gamma^+(\alpha)\cup \Gamma^-(\alpha) \cup HD_1(\alpha)$.  So by Lemma~\ref{square}, we can conclude that either
\begin{eqnarray*}
&\Gamma^-(\gamma)\cap HD_1(\alpha)\subseteq (\delta_1^{H_{\delta}} \times \{\delta\}) \cup(\{\delta\} \times \delta_1^{H_{\delta}}), &\text{ or}\\
& \Gamma^-(\gamma)\cap HD_1(\alpha)\subseteq (\delta_1'^{H_{\delta}} \times \{\delta\}) \cup(\{\delta\} \times \delta_1'^{H_{\delta}}).&
\end{eqnarray*}
We assume that the first of these possibilities is true; the other proof is analogous.

Since $\Gamma^-(\gamma) \subset \Gamma^+(\alpha)\cup \Gamma^-(\alpha) \cup HD_1(\alpha)$, our assumption forces the rows and columns of $\delta_1'^{H_{\delta}}\times \delta_1'^{H_{\delta}}$ to be disjoint from the rows and columns of all other in-neighbours of $\gamma$.  
Now $\Gamma^+(\alpha) \subseteq \delta_1'^{H_{\delta}}\times \delta_1'^{H_{\delta}}$, and 
$\gamma$ has either $0$ or $a$ in-neighbours in any of these rows or columns (Lemma~\ref{abconstant-directed} together with vertex-transitivity yield this conclusion), all of which must also be out-neighbours of $\alpha$; and these are all of the in-neighbours of $\gamma$ that are also out-neighbours of $\alpha$.  Hence
we must have $|\Gamma^-(\gamma)\cap\Gamma^+(\alpha)|=ja$ for some $j$.  But Lemma~\ref{more-k-a-b+1} tells us that $|\Gamma^-(\gamma) \cap \Gamma^+(\alpha)|=(k-2a+1)/2$.  So we have $k-2a+1=2ja$, but $a$ divides each of these values with the exception of 1, and we know $a\ge 2$ (by Lemma~\ref{abconstant-directed}), a contradiction.
\end{proof}

\end{document}